\documentclass[reqno]{amsart}
\usepackage{}%默认值为reqno，表示公式编号放在公式右边，对齐在页面右边界处；如果使用选项leqno，表示公式编号放在公式左侧。
\usepackage{stmaryrd}
\usepackage{mathrsfs}
\usepackage{cases}
\usepackage{amsfonts}
\usepackage{amssymb}
\usepackage{amsmath}
\usepackage{tikz}
\usepackage{enumerate}
\newtheorem{theorem}{Theorem}[section]
\newtheorem{lemma}[theorem]{Lemma}
\newtheorem{proposition}[theorem]{Proposition}
\newtheorem{corollary}[theorem]{Corollary}
\theoremstyle{definition}
\newtheorem{definition}[theorem]{Definition}

\newtheorem{remark}[theorem]{Remark}
\numberwithin{equation}{section}

\newcommand{\blankbox}[2]

\begin{document}
\title[Limiting weak-type behaviors]
{Limiting weak-type behaviors of some integral operators}
\author{WEICHAO GUO}
\address{School of Mathematics and Information Sciences, Guangzhou University, Guangzhou, 510006, P.R.China}
\email{weichaoguomath@gmail.com}
\author{JIANXUN HE}
\address{School of Mathematics and Information Sciences, Guangzhou University, Guangzhou, 510006, P.R.China}
\email{hejianxun@gzhu.edu.cn}
\author{HUOXIONG WU$^*$}
\address{School of Mathematical Sciences, Xiamen University,
Xiamen, 361005, P.R. China} \email{huoxwu@xmu.edu.cn}

\begin{abstract}
In this paper, we explore the limiting weak-type behaviors of some integral operators including maximal operators, singular and fractional integral operators and maximal truncated singular integrals et al. Some optimal limiting weak-type behaviors are given, which essentially improve and extend the previous results in this topics.
\end{abstract}
\thanks{$^*$Corresponding author.}
\thanks{Supported by the NSF of China (Nos.11771358, 11471041, 11701112, 11671414), the NSF of Fujian Province of China (No.2015J01025)
and the China postdoctoral Science Foundation (No. 2017M612628).}
\maketitle

\section{Introduction}
Let $V$ be a probability measure,
and $V_t$ be the dilation of $V$, defined by
\begin{equation}
  V_t(E): =V(\frac{E}{t}).
\end{equation}
As $t$ tends to zero, $V_t$ concentrates its mass at the origin.
Consider the convolution of $V_t$ and a function $\phi$, namely,
\begin{equation}
  \phi\ast V_t(x):= \int_{\mathbb{R}^n}\phi(x-y)dV_t(y).
\end{equation}
Note that this convolution can be interpreted as a weighted averages of $\phi$.
Since the concentration of $V_t$, hence in the above integral, the value $g(x)$ is assigned the full mass as $t\rightarrow 0$.
More precisely, for example if $\phi\in L^p(\mathbb{R}^n)$ and $V$ is an absolutely continuous measure (respect to Lebesgue measure),
the classical property of \textbf{approximation to the identity} shows that
\begin{equation}
  \phi\ast V_t \rightarrow \phi
\end{equation}
in the sense of $L^p$.

For the linear operator $T_{\phi}(V): =\phi\ast V$, the above argument shows that if $\phi\in L^p$,
then $\phi$ can be approached by $T_{\phi}(V_t)$ as $t\rightarrow 0^+$ in the sense of $L^p$.
However, for the nonlinear case, for instance the corresponding maximal operator defined by
\begin{equation*}
  \mathcal {M}^{\alpha}_{\phi}(\mu)(x):=\sup_{r>0}(\phi^{\alpha}_r\ast \mu)(x)=\sup_{r>0}\frac{1}{r^{n-\alpha}}\int_{\mathbb{R}^n}\phi(\frac{x-y}{r})d\mu(y),
\end{equation*}
where $\phi_r^{\alpha}(x)=\frac{1}{r^{n-\alpha}}\phi(\frac{x}{r})$,
the limiting behavior become confused as $t\rightarrow 0^+$, since it can no longer be concluded or implied
by the classical theory of \textbf{approximation to the identity}.
If $\phi=\chi_{B(0,1)}$, we write $\mathcal {M}^{\alpha}: =\mathcal {M}^{\alpha}_{\phi}$ for short.
For $\alpha=0$, we write $\mathcal {M}_{\phi}:=\mathcal {M}_{\phi}^{\alpha}$.
We also use $\mathcal {M}:= \mathcal {M}_{\phi}^{\alpha}$ for $\alpha=0$ and $\phi=\chi_{B(0,1)}$, which is the classical Hardy-Littlewood maximal operator.

For the special case $\alpha=0$ and $\phi=\chi_{B(0,1)}$,
the following theorem was firstly obtained by Janakiraman\cite{Jan_Trans_2006}.

\medskip

\textbf{Theorem A}(\cite[Theorem 3.1]{Jan_Trans_2006}).
Let $V$ be a positive measure with finite variation. Then for any fixed $\lambda>0$,
\begin{equation*}
  \lim_{t\rightarrow 0^+}|\{x\in \mathbb{R}^n: \mathcal {M}V_t(x)>\lambda\}|
  =\Big|\Big\{x\in \mathbb{R}^n: \frac{V(\mathbb{R}^n)}{|x|^n}>\lambda\Big\}\Big|.
\end{equation*}
We remark that Theorem A above is actually the essential part of \cite[Theorem 3.1]{Jan_Trans_2006}.
In order to better compare with the results in this paper,
we would like to describe the previous known results in a unified form, without any change of their essence.

A stronger conclusion was also obtained in the same paper, see \cite[Corollary 3.3]{Jan_Trans_2006} or the following.

\medskip

\textbf{Theorem A*}(\cite[Corollary 3.3]{Jan_Trans_2006}).
Let $V$ be a positive measure with finite variation. Then
\begin{equation*}
  \lim_{t\rightarrow 0^+}\Big|\Big\{x\in \mathbb{R}^n: \Big|\mathcal {M}V_t(x)-\frac{V(\mathbb{R}^n)}{|x|^n}\Big|>\lambda\Big\}\Big|=0
\end{equation*}
for every fixed $\lambda>0$.

One of the main purposes of this paper is to improve and extend Theorems A and A*. More precisely, in Section 2, we will
show a stronger limiting behavior for the more general maximal operator $M_{\phi}^{\alpha}$, which is an essential improvement of Theorem A even for the special case $\mathcal{M}$ and is optimal in some sense (see Remark \ref{remark, comparison} below). The following is one of our main results.
\begin{theorem}\label{theorem, fractional maximal operator, measure}
  Let $\alpha\in [0,n)$, $\phi(x)=\Phi(|x|)$ be a radial function such that $\sup_{r>0}\phi^{\alpha}_r(e_1)<\infty$,
  where $\Phi: [0,\infty)\rightarrow [0,\infty)$ is decreasing, $e_1=(1,0,\cdots,0)$ is the vector on the unit sphere $\mathbb{S}^{n-1}$.
  Suppose that $\|\mathcal {M}_{\phi}^{\alpha}\mu\|_{L^{\frac{n}{n-\alpha},\infty}}\lesssim \mu(\mathbb{R}^n)$ holds for all positive measure $\mu$.
  Then, for any fixed $\rho>0$,
  \begin{equation*}
    \lim_{t\rightarrow0^{+}}
    \Big\|\mathcal {M}^{\alpha}_{\phi}(V_t)(\cdot)-\sup_{r>0}\phi^{\alpha}_r(\cdot)V(\mathbb{R}^n)\Big\|_{L^{\frac{n}{n-\alpha},\infty}(\mathbb{R}^n\backslash B(0,\rho))}=0
  \end{equation*}
  for all finite positive measure $V$.
\end{theorem}

\begin{corollary}\label{corollary, fractional maximal operator, measure}
  Let $\alpha\in [0,n)$, $\phi(x)=\Phi(|x|)$ be a radial function, where $\Phi: [0,\infty)\rightarrow [0,\infty)$ is decreasing, bounded and compact supported.
  Then, for any fixed $\rho>0$,
  \begin{equation*}
    \lim_{t\rightarrow0^{+}}
    \Big\|\mathcal {M}^{\alpha}_{\phi}(V_t)(\cdot)-\sup_{r>0}\phi^{\alpha}_r(\cdot)V(\mathbb{R}^n)\Big\|_{L^{\frac{n}{n-\alpha},\infty}(\mathbb{R}^n\backslash B(0,\rho))}=0
  \end{equation*}
  for any fixed finite positive measure $V$.
\end{corollary}

\begin{remark}\label{remark, comparison}
  We would like to make some comparisons between Corollary \ref{corollary, fractional maximal operator, measure} and
  the previous results in Theorems A and A*.
  Choosing $\alpha=0$ and $\phi=\chi_{B(0,1)}$ in Corollary \ref{corollary, fractional maximal operator, measure},
  and by the fact $\sup_{r>0}\phi_r^0(x)=\frac{1}{|x|^n}$ for $\phi=\chi_{B(0,1)}$,
  we
  immediately obtain
  \begin{equation}\label{for proof, 3}
    \lim_{t\rightarrow0^{+}}
    \Big\|\mathcal {M}(V_t)(\cdot)-\frac{V(\mathbb{R}^n)}{|\cdot|^n}\Big\|_{L^{1,\infty}(\mathbb{R}^n\backslash B(0,\rho))}=0
  \end{equation}
  for any fixed positive constant $\rho$.  Then Proposition \ref{proposition, comparison for limiting behavior} shows that
  this limiting is stronger than that in Theorems A and A*.

  In fact, the convergence in (\ref{for proof, 3})
  is optimal, in the sense that there is some $V$ such that following limiting behavior is negative:
  \begin{equation}\label{for proof, 2}
    \lim_{t\rightarrow0^{+}}
    \Big\|\mathcal {M}(V_t)(\cdot)-\frac{V(\mathbb{R}^n)}{|\cdot|^n}\Big\|_{L^{1,\infty}(\mathbb{R}^n)}=0.
  \end{equation}
  More precisely, we take $dV(x)=\chi_{B(0,1)}(x)dx$, where $dx$ denotes the Lebesgue measure.
  Then
  \begin{equation*}
    \begin{split}
      \mathcal {M}V_t(x)
      =
      \sup_{r>0}\frac{1}{r^n}\int_{B(x,r)}dV_t(y)
      =
      \sup_{r>0}\frac{1}{r^n}\int_{B(x,r)}\frac{1}{t^n}\chi_{B(0,t)}(y)dy
      =
      \sup_{r>0}\frac{1}{r^n}\cdot \frac{|B(0,t)\cap B(x,r)|}{t^n}.
    \end{split}
  \end{equation*}
  If $|x|\leq t/2$, we get
    \begin{equation*}
    \begin{split}
      \mathcal {M}V_t(x)
      =
      \frac{1}{r^n}\cdot \frac{|B(0,t)\cap B(x,r)|}{t^n}\Big|_{r\leq |t|-|x|}
      =
      \frac{1}{r^n}\cdot \frac{|B(x,r)|}{t^n}=\frac{|B(0,1)|}{t^n}=\frac{V(\mathbb{R}^n)}{t^n}.
    \end{split}
  \end{equation*}
  Also, for $|x|\leq t/2$, we have
  \begin{equation*}
    \frac{V(\mathbb{R}^n)}{|x|^n}\geq \frac{4V(\mathbb{R}^n)}{t^n}.
  \end{equation*}
  Combination of the above two estimates yields that
  \begin{equation*}
    \Big|\mathcal {M}V_t(x)-\frac{V(\mathbb{R}^n)}{|x|^n}\Big|\geq \frac{3V(\mathbb{R}^n)}{t^n}\ \text{for}\ x\in B(0,t/2).
  \end{equation*}
  Choose $\lambda_0=\frac{2 V(\mathbb{R}^n)}{t^n}$. We have
  \begin{equation*}
    \begin{split}
      \lambda_0\Big|\Big\{x\in \mathbb{R}^n:\, \Big|\mathcal {M}V_t(x)-\frac{V(\mathbb{R}^n)}{|x|^n}\Big|>\lambda_0\Big\}\Big|
      \geq &
      \lambda_0\Big|\Big\{x\in B(0,t/2):\,\Big|\mathcal {M}V_t(x)-\frac{V(\mathbb{R}^n)}{|x|^n}\Big|>\lambda_0\Big\}\Big|
      \\
      = &
      \lambda_0|B(0,t/2)|=\frac{2}{t^n}\cdot V(\mathbb{R}^n)\, \frac{t^n}{2^n}\, |B(0,1)|
      =\frac{|B(0,1)|^2}{2^{n-1}}.
    \end{split}
  \end{equation*}
  Thus,
  \begin{equation*}
    \begin{split}
     \varliminf_{t\rightarrow0^{+}}
      \Big\|\mathcal {M}(V_t)(\cdot)-\frac{V(\mathbb{R}^n)}{|\cdot|^n}\Big\|_{L^{1,\infty}(\mathbb{R}^n)}
      \geq &
      \lambda_0|\{x\in \mathbb{R}^n: |\mathcal {M}V_t(x)-\mathcal {M}\delta_0(x)\, V(\mathbb{R}^n)|>\lambda_0\}|
      \\
      \geq &
      \frac{|B(0,1)|^2}{2^{n-1}}.
    \end{split}
  \end{equation*}
  This shows that the limiting in (\ref{for proof, 2}) is negative, and (\ref{for proof, 3}) is optimal.
\end{remark}

On the other hand, the limiting behavior of singular integral with homogeneous kernel was also considered in \cite{Jan_Trans_2006}. Subsequently, it was improved by Ding and Lai in \cite{DL2}. Moreover, the weak limiting behavior of maximal operator associated with homogeneous kernel was also considered in \cite{DL1}. To state the relevant previous results, we first recall several definitions and notations. The integral operator we are interested in this paper are of the form
\begin{equation}\label{for proof, 4}
  T_{\Omega}^{\alpha}\mu(x):=\int_{\mathbb{R}^n}\frac{\Omega(x-y)}{|x-y|^{n-\alpha}}d\mu(y),
\end{equation}
where $\alpha\in [0,n)$,
$\Omega$ is a homogeneous function of degree zero and satisfies the following mean value zero property when $\alpha= 0$:
\begin{equation}\label{mean value zero}
  \int_{\mathbb{S}^{n-1}}\Omega(x')d\sigma(x')=0.
\end{equation}
As usual for $\alpha=0$, the right side of the equation (\ref{for proof, 4}) should be represented in the sense of principal value.
We write $T_{\Omega}:=T_{\Omega}^0$ for short.

We also consider the corresponding maximal operator associated with homogeneous kernel $\Omega$ defined by
\begin{equation*}
  M_{\Omega}^{\alpha}\mu(x):=\sup_{r>0}\frac{1}{r^{n-\alpha}}\int_{B(x,r)}|\Omega(x-y)|d\mu(y)
\end{equation*}
Note that if we take $\phi=|\Omega|\chi_{B(0,1)}$, then
\begin{equation*}
  \begin{split}
    \mathcal {M}_{\phi}^{\alpha}V(x)
    = &
    \sup_{r>0}\frac{1}{r^{n-\alpha}}\int_{\mathbb{R}^n}\Big|\Omega(\frac{x-y}{r})\Big|\chi_{B(0,1)}(\frac{x-y}{r})dV(y)
    \\
    = &
    \sup_{r>0}\frac{1}{r^{n-\alpha}}\int_{B(x,r)}|\Omega(x-y)|dV(y)=M_{\Omega}^{\alpha}V(x).
  \end{split}
\end{equation*}
In this case, $\phi=|\Omega|\chi_{B(0,1)}$ is not a radial function anymore.
To keep the limiting behavior still valid in this case,
we need to add some angular regularity to $\Omega$, which can be viewed as the alternative condition of radial property.
Now, we give the definition of angular regularity, namely, the Dini-condition.

\begin{definition}[$L^q_s$-Dini condition]\label{definition, Dini condition}
  Suppose $\Omega$ is a homogeneous function of degree zero. Let $1\leq q\leq \infty, 0\leq s<n$. We say that
  $\Omega$ satisfies $L^q_{s}$-Dini condition if
  \begin{enumerate}
    \item $\Omega\in L^q(\mathbb{S}^{n-1})$,
    \item $\int_0^1 \frac{{\omega}_q(t)}{t^{1+s}}dt<\infty$,
  \end{enumerate}
  where ${\omega}_q$ is called the (modified) integral continuous modulus of $\Omega$ of degree $q$, defined by
  \begin{equation}
    {\omega}_q(t):=\left(\sup_{|h|\leq t}\int_{\mathbb{S}^{n-1}}|\Omega(x'+h)-\Omega(x')|^qd\sigma(x')\right)^{1/q}.
  \end{equation}
\end{definition}
For brevity, we use $L^q$-Dini instead of $L^q_0$-Dini.
Note that
%the $L^q$-Dini condition in Definition \ref{definition, Dini condition}
%is coincide with the classical definition of $L^q$-Dini condition.
%We also remark that
  for $q>1$ the above definition is a little different
  from \cite[Definition 2.7]{DL2} and \cite[Definition 4.1]{DL1}.

In order to compare our results with the relevant previous results, we now list the main theorems in \cite{DL1,DL2} as follows.

\medskip

\textbf{Theorem B} (cf. \cite{DL1}).\quad
Let $\Omega$ be a homogeneous function of degree $0$, satisfying (\ref{mean value zero}) and $L^1$-Dini condition.
Then for any fixed $\lambda>0$,
\begin{equation*}
  \lim_{t\rightarrow 0^+}|\{x\in \mathbb{R}^n: M_{\Omega}V_t(x)>\lambda \}|
  =\Big|\Big\{x\in \mathbb{R}^n: \frac{|\Omega(x)|V(\mathbb{R}^n)}{|x|^n}>\lambda\Big\}\Big|
\end{equation*}
for all finite positive absolutely continuous measure $V$.

\medskip

\textbf{Theorem C} (cf. \cite{DL1}).\quad
Let $\Omega\in L^{\frac{n}{n-\alpha}}(\mathbb{S}^{n-1})$ be a homogeneous function of degree $0$, satisfying $L^1_{\alpha}$-Dini condition.
Then for any fixed $\lambda>0$,
\begin{equation*}
  \lim_{t\rightarrow 0^+}|\{x\in \mathbb{R}^n: M_{\Omega}^{\alpha}V_t(x)>\lambda \}|
  =\Big|\Big\{x\in \mathbb{R}^n: \frac{|\Omega(x)|V(\mathbb{R}^n)}{|x|^{n-\alpha}}>\lambda \Big\}\Big|
\end{equation*}
for all finite positive absolutely continuous measure $V$.

\medskip

\textbf{Theorem D} (cf. \cite{DL2}).\quad
Let $\Omega$ be a homogeneous function of degree $0$, satisfying (\ref{mean value zero}) and $L^1$-Dini condition.
Then for any fixed $\lambda>0$,
\begin{equation*}
  \lim_{t\rightarrow 0^+}\Big|\Big\{x\in \mathbb{R}^n: T_{\Omega}V_t(x)>\lambda \Big\}\Big|
  =|\{x\in \mathbb{R}^n: \frac{|\Omega(x)|V(\mathbb{R}^n)}{|x|^n}>\lambda \}|
\end{equation*}
for all finite positive absolutely continuous measure $V$.

\medskip

\textbf{Theorem E} (cf. \cite{DL2}).\quad
Let $\Omega\in L^{\frac{n}{n-\alpha}}(\mathbb{S}^{n-1})$ be a homogeneous function of degree $0$, satisfying $L^1_{\alpha}$-Dini condition.
Then for any fixed $\lambda>0$,
\begin{equation*}
  \lim_{t\rightarrow 0^+}|\{x\in \mathbb{R}^n: T_{\Omega}^{\alpha}V_t(x)>\lambda \}|
  =\Big|\Big\{x\in \mathbb{R}^n: \frac{|\Omega(x)|V(\mathbb{R}^n)}{|x|^{n-\alpha}}>\lambda \Big\}\Big|
\end{equation*}
for all finite positive absolutely continuous measure $V$.

\medskip

The second purpose of this paper is to improve and extend the above results in Theorems B-E.
$\Omega\in L^{\frac{n}{n-\alpha}}(\mathbb{S}^{n-1})$ will be proved to be necessary
if the corresponding operator $M_{\Omega}^{\alpha}$ or $T_{\Omega}^{\alpha}$ is bounded
with $\Omega$ satisfying $L_{1}^{\alpha}$-Dini condition.
The limiting behaviors in the above four theorems will be improved (see Remarks \ref{remark, comparison, maximal operator with homogeneous kernel} and \ref{remark, comparison, integral operator with homogeneous kernel} below). Our main results in this part can be formulated as follows,

\begin{theorem}\label{Theorem, fractional maximal with homogeneous kernel}
Let $\alpha\in [0,n)$, $V$ be an absolutely continuous positive measure.
Suppose that $\Omega$ is a homogeneous function of degree zero and satisfies the $L^1_{\alpha}$-Dini condition.
If the maximal operator $M_{\Omega}^{\alpha}$ is bounded from $L^1$ to $L^{\frac{n}{n-\alpha},\infty}$, then
\begin{enumerate}
\item    $\frac{\Omega(x)}{|x|^{n-\alpha}}\in L^{\frac{n}{n-\alpha},\infty}$, $\Big\|\frac{\Omega(\cdot)}{|\cdot|^{n-\alpha}}\Big\|_{L^{\frac{n}{n-\alpha},\infty}}\lesssim \|M_{\Omega}^{\alpha}\|_{L^1\rightarrow L^{\frac{n}{n-\alpha},\infty}}$;
\item    $\Omega\in L^{\frac{n}{n-\alpha}}(\mathbb{S}^{n-1})$, $\|\Omega \|_{L^{\frac{n}{n-\alpha}}(\mathbb{S}^{n-1})}\lesssim \|M_{\Omega}^{\alpha}\|_{L^1\rightarrow L^{\frac{n}{n-\alpha},\infty}}$;
\item   $\displaystyle\lim_{t\rightarrow 0^{+}}\Big|\Big\{x\in \mathbb{R}^n: \Big|M^{\alpha}_{\Omega}(V_t)(\cdot)-\frac{|\Omega(\cdot)|}{|\cdot|^{\frac{n}{n-\alpha}}}V(\mathbb{R}^n)\Big|>\lambda\Big\}\Big|=0$, \, $\forall\,\lambda>0$.
\end{enumerate}
\end{theorem}

\begin{theorem}\label{Theorem, fractional maximal with homogeneous kernel, Lq-Dini}
Let $\alpha\in [0,n)$, $V$ be a absolutely continuous positive measure.
Suppose that $\Omega$ is a homogeneous function of degree zero and satisfies the $L^{\frac{n}{n-\alpha}}$-Dini condition.
If the maximal operator $M_{\Omega}^{\alpha}$ is bounded from $L^1$ to $L^{\frac{n}{n-\alpha},\infty}$, then
\begin{enumerate}
\item    $\frac{\Omega(x)}{|x|^{n-\alpha}}\in L^{\frac{n}{n-\alpha},\infty}$, $\Big\|\frac{\Omega(\cdot)}{|\cdot|^{n-\alpha}}\Big\|_{L^{\frac{n}{n-\alpha},\infty}}\lesssim \|M_{\Omega}^{\alpha}\|_{L^1\rightarrow L^{\frac{n}{n-\alpha},\infty}}$;
\item    $\Omega\in L^{\frac{n}{n-\alpha}}(\mathbb{S}^{n-1})$, $\|\Omega \|_{L^{\frac{n}{n-\alpha}}(\mathbb{S}^{n-1})}\lesssim \|M_{\Omega}^{\alpha}\|_{L^1\rightarrow L^{\frac{n}{n-\alpha},\infty}}$;
\item    $\displaystyle\lim_{t\rightarrow 0^{+}} \Big\|M^{\alpha}_{\Omega}(V_t)-\frac{|\Omega(\cdot)|}{|\cdot|^{n-\alpha}}V(\mathbb{R}^n)\Big\|_{L^{\frac{n}{n-\alpha},\infty}(\mathbb{R}^n\backslash B(0,\rho))}=0$
    for every $\rho>0$.
\end{enumerate}
\end{theorem}

\begin{remark}\label{remark, comparison, maximal operator with homogeneous kernel}
  In Theorem \ref{Theorem, fractional maximal with homogeneous kernel}, we show that $\Omega\in L^{\frac{n}{n-\alpha}}(\mathbb{S}^{n-1})$
  is necessary if $M_{\Omega}^{\alpha}$ is $L^1\rightarrow L^{\frac{n}{n-\alpha},\infty}$ bounded
  with $\Omega$ satisfying the $L^1_{\alpha}$-Dini condition.
  This conclusion is new for $\alpha>0$,
  since $\Omega\in L^{\frac{n}{n-\alpha}}(\mathbb{S}^{n-1})$ is needed in the assumption of Theorem C.
  Moreover, compared with the conclusion in Theorems B and C, in Theorem \ref{Theorem, fractional maximal with homogeneous kernel}
  we obtain the limiting behavior in type-2 sense, which is stronger than type-3 sense as in Theorems B and C.
  Furthermore, as mentioned in Theorem \ref{Theorem, fractional maximal with homogeneous kernel, Lq-Dini},
  if $\Omega$ satisfies $L^{\frac{n}{n-\alpha}}$-Dini condition,
  the limiting behavior holds in the type-1 sense (see Remark 2.2 in Section 2 for the sense of "type-$i$", $i=1,\,2,\,3$).
\end{remark}

\begin{theorem}\label{Theorem, fractional operator with homogeneous kernel}
Let $\alpha\in [0,n)$, $V$ be a absolutely continuous positive measure.
Suppose that $\Omega$ is a homogeneous function of degree zero and satisfies the $L^1_{\alpha}$-Dini condition.
If the singular or factional integral operator $T_{\Omega}^{\alpha}$ is bounded from $L^1$ to $L^{\frac{n}{n-\alpha},\infty}$, then
\begin{enumerate}
\item    $\frac{\Omega(x)}{|x|^{n-\alpha}}\in L^{\frac{n}{n-\alpha},\infty}$, $\Big\|\frac{\Omega(\cdot)}{|\cdot|^{n-\alpha}}\Big\|_{L^{\frac{n}{n-\alpha},\infty}}\lesssim \|T_{\Omega}^{\alpha}\|_{L^1\rightarrow L^{\frac{n}{n-\alpha},\infty}}$;
\item    $\Omega\in L^{\frac{n}{n-\alpha}}(\mathbb{S}^{n-1})$, $\|\Omega \|_{L^{\frac{n}{n-\alpha}}(\mathbb{S}^{n-1})}\lesssim \|T_{\Omega}^{\alpha}\|_{L^1\rightarrow L^{\frac{n}{n-\alpha},\infty}}$;
\item   $\displaystyle\lim_{t\rightarrow 0^{+}}\Big|\Big\{x\in \mathbb{R}^n: \Big|T^{\alpha}_{\Omega}(V_t)(\cdot)-\frac{\Omega(\cdot)}{|\cdot|^{\frac{n}{n-\alpha}}}V(\mathbb{R}^n)\Big|>\lambda\Big\}\Big|=0$, $\forall\,\lambda>0$.
\end{enumerate}
\end{theorem}

\begin{theorem}\label{Theorem, fractional operator with homogeneous kernel, Lq-Dini}
Let $\alpha\in [0,n)$, $V$ be a absolutely continuous positive measure.
Suppose that $\Omega$ is a homogeneous function of degree zero and satisfies the $L^{\frac{n}{n-\alpha}}$-Dini condition.
If $T_{\Omega}^{\alpha}$ is bounded from $L^1$ to $L^{\frac{n}{n-\alpha},\infty}$, then
\begin{enumerate}
\item    $\frac{\Omega(x)}{|x|^{n-\alpha}}\in L^{\frac{n}{n-\alpha},\infty}$, $\Big\|\frac{\Omega(\cdot)}{|\cdot|^{n-\alpha}}\Big\|_{L^{\frac{n}{n-\alpha},\infty}}\lesssim \|T_{\Omega}^{\alpha}\|_{L^1\rightarrow L^{\frac{n}{n-\alpha},\infty}}$;
\item    $\Omega\in L^{\frac{n}{n-\alpha}}(\mathbb{S}^{n-1})$, $\|\Omega \|_{L^{\frac{n}{n-\alpha}}(\mathbb{S}^{n-1})}\lesssim \|T_{\Omega}^{\alpha}\|_{L^1\rightarrow L^{\frac{n}{n-\alpha},\infty}}$;
\item    $\displaystyle\lim_{t\rightarrow 0^{+}} \Big\|T^{\alpha}_{\Omega}(V_t)-\frac{\Omega(\cdot)}{|\cdot|^{n-\alpha}}V(\mathbb{R}^n)\Big\|_{L^{\frac{n}{n-\alpha},\infty}(\mathbb{R}^n\backslash B(0,\rho))}=0$
    for every $\rho>0$.
\end{enumerate}
\end{theorem}

\begin{remark}\label{remark, comparison, integral operator with homogeneous kernel}
Theorem \ref{Theorem, fractional operator with homogeneous kernel} shows that $\Omega\in L^{\frac{n}{n-\alpha}}$
is necessary if $T_{\Omega}^{\alpha}$ is $L^1\rightarrow L^{\frac{n}{n-\alpha},\infty}$ bounded
with $\Omega$ satisfying the $L^1_{\alpha}$-Dini condition.
Moreover, we obtain the corresponding limiting behavior in type-2 sense,
which is better than the previous results in Theorems D and E.
Furthermore, we establish the type-1 convergence if $\Omega$ satisfies the $L^{\frac{n}{n-\alpha}}$-Dini condition.
\end{remark}

\begin{remark}
From Theorems \ref{Theorem, fractional maximal with homogeneous kernel} to \ref{Theorem, fractional operator with homogeneous kernel},
one can find that the limiting behaviors of $M_{\Omega}^{\alpha}V_t$ and $T_{\Omega}^{\alpha}V_t$ are different.
More precisely, in the sense of type-2 (or type-3),
the limit of $M_{\Omega}^{\alpha}V_t$ is $\frac{|\Omega(x)|}{|x|^{n-\alpha}}V(\mathbb{R}^n)$,
but that of $T_{\Omega}^{\alpha}V_t$ is $\frac{\Omega(x)}{|x|^{n-\alpha}}V(\mathbb{R}^n)$.
Note that this phenomenon can not be observed from type-3 convergence as mentioned in Theorems B-E.
\end{remark}

\begin{remark}
Observe that $\sup_{r>0}\phi^{\alpha}_r(x)=\frac{\Omega(x)}{|x|^{n-\alpha}}$
if $\phi=|\Omega|\chi_{B(0,1)}$.
Thus, Theorem \ref{Theorem, fractional maximal with homogeneous kernel} actually has the same form as Theorem \ref{theorem, fractional maximal operator, measure}.
\end{remark}

Furthermore, as corollary, the following result gives a partial answer for
why the integral index $\frac{n}{n-\alpha}$ is optimal
in the study of boundedness of the fractional integral operators with homogeneous kernel.

\begin{corollary}\label{corollary, equivalent}
Suppose $\alpha\in (0,n)$, $\Omega$ satisfies the $L^1_{\alpha}$-Dini condition. Then the following three statements are equivalent:
\begin{enumerate}
  \item $\Omega\in L^{\frac{n}{n-\alpha}}(\mathbb{S}^{n-1})$.
  \item $M^{\alpha}_{\Omega}$ is bounded from $L^1(\mathbb{R}^n)$ to $L^{\frac{n}{n-\alpha},\infty}(\mathbb{R}^n)$.
  \item $T^{\alpha}_{\Omega}$ is bounded from $L^1(\mathbb{R}^n)$ to $L^{\frac{n}{n-\alpha},\infty}(\mathbb{R}^n)$.
\end{enumerate}
\end{corollary}

This paper is organized as follows. In Section 2, we deal with the limiting behaviors for a wide class of maximal functions.
The limiting behaviors for the maximal operators associated with homogeneous kernels will be considered in Section 3.
Next, we establish the limiting behaviors of singular and fractional integral operators and the maximal truncated singular integrals in Section 4.
Finally, we give some limiting behaviors for the general convolution operators in Section 5.

\section{Maximal operator associated with radial functions}
In order to distinguish the various kinds of limiting behaviors, we first establish the following proposition.
In this paper, all the limiting behaviors can be compared each other in the framework of this proposition.

\begin{proposition}\label{proposition, comparison for limiting behavior}
  Let $0<p<\infty$. Suppose that $f\in L^{p,\infty}(\mathbb{R}^n)$, and
  $|\{x\in \mathbb{R}^n: |f(x)|=\lambda\}|=0$ for all $\lambda > 0$.
  Let $\{f_{(t)}\}_{t>0}$ be a sequence of measurable functions.
  Then for the following three statements:
  \begin{enumerate}
  \item $\forall\varepsilon>0$, $\exists A_{\varepsilon}\subset \mathbb{R}^n$, s.t.,
  $|A_{\varepsilon}|<\varepsilon$ and $\displaystyle\lim_{t\rightarrow 0^{+}}\|f-f_{(t)}\|_{L^{p,\infty}(\mathbb{R}^n\backslash A_{\varepsilon})}=0$.
  \item $\displaystyle\lim_{t\rightarrow 0^{+}}|\{x\in \mathbb{R}^n:\, |f_{(t)}(x)-f(x)|>\lambda\}|=0$,\, $\forall\,\lambda>0$,
  \item $\displaystyle\lim_{t\rightarrow 0^{+}}|\{x\in \mathbb{R}^n: \,|f_{(t)}(x)|>\lambda\}|=|\{x\in \mathbb{R}^n: |f(x)|>\lambda\}|$,\, $\forall\,\lambda>0$,
\end{enumerate}
we have
\begin{equation*}
  (1)\Rightarrow (2) \Rightarrow (3),\ (3)\nRightarrow (2),\ (2)\nRightarrow (1).
\end{equation*}
\end{proposition}
\begin{proof}
We first verify $(1)\Rightarrow (2)$.
For any $\varepsilon>0$, there exists a set $|A_{\varepsilon}|<\varepsilon$ such that
$$\lim_{t\rightarrow 0^{+}}\|f-f_{(t)}\|_{L^{p,\infty}(\mathbb{R}^n\backslash A_{\varepsilon})}=0.$$
This implies that
  \begin{equation*}
    \begin{split}
      |\{x\in \mathbb{R}^n: |f_{(t)}(x)-f(x)|>\lambda\}|
     & \leq
      |\{x\in A^c_{\varepsilon}: |f_{(t)}(x)-f(x)|>\lambda\}|+|A_{\varepsilon}|
      \\
      &<
      \frac{\|f_{(t)}-f\|^p_{L^{p,\infty}(\mathbb{R}^n\backslash A_{\epsilon})}}{\lambda^p}+|A_{\varepsilon}|.
    \end{split}
  \end{equation*}
Thus,
  \begin{equation*}
     \varlimsup_{t\rightarrow 0^{+}}|\{x\in \mathbb{R}^n: |f_{(t)}(x)-f(x)|>\lambda\}|
    \leq
    |A_{\varepsilon}|<\varepsilon.
  \end{equation*}
  By the arbitrary of $\varepsilon$, we obtain conclusion (2).

Next, we show that $(2)\Rightarrow (3)$. For a small constant $\nu\in (0,1)$, we have
  \begin{equation*}
    \begin{split}
      |\{x\in \mathbb{R}^n: |f_{(t)}(x)|>\lambda\}|
            \leq
      |\{x\in \mathbb{R}^n: |f(x)|>(1-\nu)\lambda\}|+
      |\{x\in \mathbb{R}^n: |f_{(t)}(x)-f(x)|>\nu\lambda\}|.
    \end{split}
  \end{equation*}
  By (2), we have $\lim_{t\rightarrow 0^{+}}|\{x\in \mathbb{R}^n: |f_{(t)}(x)-f(x)|>\nu\lambda\}|=0$.
  This implies that
  \begin{equation*}
    \varlimsup_{t\rightarrow 0}|\{x\in \mathbb{R}^n: |f_{(t)}(x)|>\lambda\}|
    \leq|\{x\in \mathbb{R}^n: |f(x)|>(1-\nu)\lambda\}|.
  \end{equation*}
  Letting $\nu\rightarrow 0$, we have
  \begin{equation}\label{proposition, for proof 1}
    \begin{split}
      \varlimsup_{t\rightarrow 0}|\{x\in \mathbb{R}^n: |f_{(t)}(x)|>\lambda\}|
      \leq &|\{x\in \mathbb{R}^n: |f(x)|\geq \lambda\}|
      \\
      = & |\{x\in \mathbb{R}^n: |f(x)|> \lambda\}|,
    \end{split}
  \end{equation}
  where in the last equality we use the fact $|\{x\in \mathbb{R}^n: |f(x)|= \lambda\}|=0$.

  On the other hand,
  \begin{equation*}
    \begin{split}
      |\{x\in \mathbb{R}^n: |f_{(t)}(x)|>\lambda\}|
            \geq
      |\{x\in \mathbb{R}^n: |f(x)|>(1+\nu)\lambda\}|
      -
      |\{x\in \mathbb{R}^n: |f_{(t)}(x)-f(x)|>\nu\lambda\}|
    \end{split}
  \end{equation*}
  Using the assumption (2), we obtain that
  \begin{equation*}
    \varliminf_{t\rightarrow 0}|\{x\in \mathbb{R}^n: |f_{(t)}(x)|>\lambda\}|
    \geq |\{x\in \mathbb{R}^n: |f(x)|>(1+\nu)\lambda\}|.
  \end{equation*}
  Letting $\nu\rightarrow 0$, we have
  \begin{equation}\label{proposition, for proof 2}
    \begin{split}
      \varliminf_{t\rightarrow 0}|\{x\in \mathbb{R}^n: |f_{(t)}(x)|>\lambda\}|
      \geq  |\{x\in \mathbb{R}^n: |f(x)|> \lambda\}|.
    \end{split}
  \end{equation}
  The combination of (\ref{proposition, for proof 1}) and (\ref{proposition, for proof 2}) yields the desired conclusion (3).

  The proof of that $(3)\nRightarrow (2)$ is simple, we omit the detail here.

  Finally, we show that $(2)\nRightarrow (1)$. Let $g(x)=|x|^{-n/p}$, $g_{(t)}(x)=g(x)\chi_{B(0,t^{-1})}(x)$.
  Obviously, $g\in L^{p,\infty}(\mathbb{R}^n)$. For any fixed $\lambda>0$, we have
  \begin{equation*}
    \{x\in \mathbb{R}^n: |g(x)-g_{(t)}(x)|>\lambda\}=\emptyset\, \ \text{for sufficient small}\ t.
  \end{equation*}
  Thus, $$\lim_{t\rightarrow 0^{+}}|\{x\in \mathbb{R}^n: |g_{(t)}(x)-g(x)|>\lambda\}|=0.$$
  However, for any $\epsilon>0$, $|A_{\epsilon}|<\epsilon$, $t>0$, if we take $\lambda$ sufficiently small, then
  \begin{equation*}
    |\{x\in \mathbb{R}^n: |g(x)-g_{(t)}(x)|>\lambda\}|
    = \{x\in B^c(0,t^{-1}): |x|^{-n/p}>\lambda\}
    \sim \lambda^{-p}.
  \end{equation*}
  Consequently, as $\lambda\rightarrow 0$,
  \begin{equation*}
    |\{x\in A^c_{\epsilon}: |g(x)-g_{(t)}(x)|>\lambda\}|\gtrsim |\{x\in \mathbb{R}^n: |g(x)-g_{(t)}(x)|>\lambda\}|-|A_{\epsilon}|\sim \lambda^{-p}.
  \end{equation*}
  This implies that
  \begin{equation*}
    \|g-g_{(t)}\|_{L^{p,\infty}(\mathbb{R}^n\backslash A_{\varepsilon})}
    =
    \sup_{\lambda>0}\lambda\cdot |\{x\in A^c_{\epsilon}: |g(x)-g_{(t)}(x)|>\lambda\}|^{1/p}
    \gtrsim 1
    \ \text{for every}\ t>0,
  \end{equation*}
  which leads to a contradiction with (1). Proposition \ref{proposition, comparison for limiting behavior} is proved.
\end{proof}

\begin{remark}
  For the sake of brevity, for $i=1,\,2,\,3$, we say that a sequence of functions $f_{(t)}$ tends to $f$
  in the type-$i$ sense,
  if (i) is valid as in Proposition \ref{proposition, comparison for limiting behavior}.
\end{remark}

\medskip

{\it Proof of Theorem \ref{theorem, fractional maximal operator, measure}.}\quad
Without loss of generality, we may assume $V$ is a probability measure, that is, $V(\mathbb{R}^n)=1$.
  For $t>0$, denote
\begin{equation*}
  dV^1_t:=dV_t\chi_{B(0,r_t)},\hspace{6mm}dV^2_t:=dV_t\chi_{B^c(0,r_t)}
\end{equation*}
where $r_t=\sqrt{t}$. By the definition of $V_t$, we obtain
\begin{equation*}
  V^2_t(\mathbb{R}^n)=V_t(B^c(0,r_t))=1-V_t(B(0,r_t))=1-V(B(0,t^{-1/2}))=: \epsilon_t \rightarrow 0^+
\end{equation*}
as $t\rightarrow 0^+$.
Then $V^1_t(\mathbb{R}^n)=1-\epsilon_t$. By the quasi-triangle inequality for $L^{\frac{n}{n-\alpha},\infty}$,
and using the boundedness of $M^{\alpha}_{\phi}$,
we deduce that
\begin{equation}\label{l2.3}
  \begin{split}
    \|M^{\alpha}_{\phi}(V_t)(\cdot)&-\sup_{r>0}\phi^{\alpha}_r(\cdot)\|_{L^{\frac{n}{n-\alpha},\infty}(\mathbb{R}^n\backslash B(0,\rho))}
    \\
    &\lesssim
    \|\mathcal {M}^{\alpha}_{\phi}(V^1_t)(\cdot)-\sup_{r>0}\phi^{\alpha}_r(\cdot)\|_{L^{\frac{n}{n-\alpha},\infty}(\mathbb{R}^n\backslash B(0,\rho))}
    +\|\mathcal {M}^{\alpha}_{\phi}(V^2_t)(\cdot)\|_{L^{\frac{n}{n-\alpha},\infty}(\mathbb{R}^n\backslash B(0,\rho))}
    \\
    &\lesssim
    \|\mathcal {M}^{\alpha}_{\phi}(V^1_t)(\cdot)-\sup_{r>0}\phi^{\alpha}_r(\cdot)\|_{L^{\frac{n}{n-\alpha},\infty}(\mathbb{R}^n\backslash B(0,\rho))}
    +V^2_t(\mathbb{R}^n)
    \\
   &=
    \|\mathcal {M}^{\alpha}_{\phi}(V^1_t)(\cdot)-\sup_{r>0}\phi^{\alpha}_r(\cdot)\|_{L^{\frac{n}{n-\alpha},\infty}(\mathbb{R}^n\backslash B(0,\rho))}
    +\epsilon_t.
  \end{split}
\end{equation}
Set
\begin{equation*}
  A_{t,\rho}^{\lambda}:=\Big\{x\in B^c(0,\rho): \Big|\mathcal {M}^{\alpha}_{\phi}(V^1_t)(x)-\sup_{r>0}\phi_r(x)\Big|>\lambda\Big\}.
\end{equation*}
For $x\in A_{t,\rho}^{\lambda}$ and sufficient small $t$ such that $r_t<\rho/2$, we have
\begin{equation*}
  \begin{split}
    \mathcal {M}_{\phi}(V^1_t)(x)-\sup_{r>0}\phi^{\alpha}_r(x)
    = &
    \sup_{r>0}\frac{1}{r^{n-\alpha}}\int_{\mathbb{R}^n}\phi(\frac{x-y}{r})dV^1_t(y)-\sup_{r>0}\phi^{\alpha}_r(x)
    \\
    \leq &
    \sup_{r>0}\frac{1}{r^{n-\alpha}}\int_{\mathbb{R}^n}\phi(\frac{|x|-r_t}{r})dV^1_t(y)-\sup_{r>0}\phi^{\alpha}_r(x)
    \\
    = &
    \left(\frac{\rho}{\rho-r_t}\right)^{n-\alpha}\sup_{r>0}\frac{1}{r^{n-\alpha}}\int_{\mathbb{R}^n}\phi(\frac{x}{r})dV^1_t(y)
    -\sup_{r>0}\phi^{\alpha}_r(x)
    \\
    \leq &
    \left(\left(\frac{\rho}{\rho-r_t}\right)^{n-\alpha}-1\right)\sup_{r>0}\phi^{\alpha}_r(x).
  \end{split}
\end{equation*}
Also, for the oppositive direction we have
\begin{equation}
  \begin{split}
    \mathcal {M}_{\phi}(V^1_t)(x)-\sup_{r>0}\phi^{\alpha}_r(x)
    = &
    \sup_{r>0}\frac{1}{r^{n-\alpha}}\int_{\mathbb{R}^n}\phi(\frac{x-y}{r})dV^1_t(y)-\sup_{r>0}\phi^{\alpha}_r(x)
    \\
    \geq &
    \sup_{r>0}\frac{1}{r^{n-\alpha}}\int_{\mathbb{R}^n}\phi(\frac{|x|+r_t}{r})dV^1_t(y)-\sup_{r>0}\phi^{\alpha}_r(x)
    \\
    = &
    \left(\frac{|x|}{|x|+r_t}\right)^{n-\alpha}\sup_{r>0}\frac{1}{r^{n-\alpha}}\int_{\mathbb{R}^n}\phi(\frac{x}{r})dV^1_t(y)
    -\sup_{r>0}\phi^{\alpha}_r(x)
    \\
    \geq &
    \left(\left(\frac{\rho}{\rho+r_t}\right)^{n-\alpha}V_t^1(\mathbb{R}^n)-1\right)\sup_{r>0}\phi^{\alpha}_r(x)
    \\
    = &
    \left(\left(\frac{\rho}{\rho+r_t}\right)^{n-\alpha}(1-\epsilon_t)-1\right)\sup_{r>0}\phi^{\alpha}_r(x)
  \end{split}
\end{equation}
Thus, for $x\in A_{t,\rho}^{\lambda}$,
\begin{equation*}
  \Big|\mathcal {M}_{\phi}(V_t)(x)-\sup_{r>0}\phi^{\alpha}_r(x)\Big|\leq \beta_t\sup_{r>0}\phi^{\alpha}_r(x),
\end{equation*}
where $\beta_t\rightarrow 0^+$ as $t\rightarrow 0^+$.
This implies that
\begin{equation}\label{l2.4}
  \begin{split}
   \Big\|\mathcal {M}^{\alpha}_{\phi}(V^1_t)(\cdot)-\sup_{r>0}\phi^{\alpha}_r(\cdot)\Big\|_{L^{\frac{n}{n-\alpha},\infty}(\mathbb{R}^n\backslash B(0,\rho))}
    \leq
    \beta_t\|\sup_{r>0}\phi^{\alpha}_r\|_{L^{\frac{n}{n-\alpha},\infty}}.
  \end{split}
\end{equation}
Note that
\begin{equation*}
  \begin{split}
    \sup_{r>0}\phi^{\alpha}_r(x)
   & =     \sup_{r>0}\frac{1}{r^{n-\alpha}}\phi(\frac{x}{r})
   =
\frac{1}{|x|^{n-\alpha}}\cdot \sup_{r>0}\frac{1}{(r/|x|)^{n-\alpha}}\phi(\frac{x'}{r/|x|})
    \\
    &=
    \frac{1}{|x|^{n-\alpha}}\sup_{r>0}\phi^{\alpha}_r(x')
    =
    \frac{1}{|x|^{n-\alpha}}\sup_{r>0}\phi^{\alpha}_r(e_1).
  \end{split}
\end{equation*}
We have
\begin{equation*}
  \begin{split}
    \beta_t\Big\|\sup_{r>0}\phi^{\alpha}_r\Big\|_{L^{\frac{n}{n-\alpha},\infty}}
    =
    \beta_t\sup_{r>0}\phi^{\alpha}_r(e_1)\left\|\frac{1}{|\cdot|^{n-\alpha}}\right\|_{L^{\frac{n}{n-\alpha},\infty}}
    \lesssim \beta_t.
  \end{split}
\end{equation*}
This together with (\ref{l2.3}) and (\ref{l2.4}) yields that
\begin{equation*}
  \begin{split}
    \Big\|\mathcal {M}^{\alpha}_{\varphi}(V_t)(\cdot)-&\sup_{r>0}\phi^{\alpha}_r(\cdot)\Big\|_{L^{\frac{n}{n-\alpha},\infty}(\mathbb{R}^n\backslash B(0,\rho))}
    \\
    &\lesssim
    \Big\|\mathcal {M}^{\alpha}_{\varphi}(V^1_t)(\cdot)-\sup_{r>0}\phi^{\alpha}_r(\cdot)\Big\|_{L^{\frac{n}{n-\alpha},\infty}(\mathbb{R}^n\backslash B(0,\rho))}
    +\epsilon_t
    \lesssim
    \beta_t+\epsilon_t\rightarrow 0 \ \text{as}\ t\rightarrow 0^+,
  \end{split}
\end{equation*}
which is the desired conclusion and completes the proof of Theorem \ref{theorem, fractional maximal operator, measure}. $\hfill\Box$

\medskip

{\it Proof of Corollary \ref{corollary, fractional maximal operator, measure}}.\quad
Without loss of generality, we assume $\text{supp} \phi\subset B(0,1)$ and $\|\phi\|_{L^{\infty}}=1$. Then
  \begin{equation*}
    \sup_{r>0}\phi^{\alpha}_r(e_1)
    =
    \sup_{r>0}\frac{1}{r^{n-\alpha}}\phi(\frac{e_1}{r})
    =
    \sup_{r>1}\frac{1}{r^{n-\alpha}}\phi(\frac{e_1}{r})
    \leq
    1.
  \end{equation*}
In order to use Theorem \ref{theorem, fractional maximal operator, measure},
we only need to verify
$\|\mathcal {M}_{\phi}^{\alpha}\mu\|_{L^{\frac{n}{n-\alpha},\infty}}\lesssim \mu(\mathbb{R}^n)$ for all positive measure $\mu$.
Since $M_{\phi}^{\alpha}\mu \lesssim M^{\alpha}\mu$, it suffices to show that
\begin{equation}\label{for proof, 1}
  \|\mathcal {M}^{\alpha}\mu\|_{L^{\frac{n}{n-\alpha},\infty}}\lesssim \mu(\mathbb{R}^n).
\end{equation}
Write $A_{\lambda}:=\{x\in \mathbb{R}^n: M^{\alpha}\mu(x)>\lambda\}$.
For any $x\in A_{\lambda}$, we can find a ball $B(x,r_x)$, satisfying that
\begin{equation*}
  \frac{1}{r_x^{n-\alpha}}\int_{B(x,r_x)}d\mu(y)>\lambda.
\end{equation*}
This implies that $|B(x,r_x)|^{\frac{n-\alpha}{n}}\lesssim \frac{1}{\lambda}\int_{B(x,r_x)}d\mu(y)$.
Obviously, $A_{\lambda}\subset \bigcup_{x\in A_{\lambda}}B(x,r_x)$.
Using the Wiener Covering Lemma, there exists disjoint collection of such balls $B_i=B(x_i,r_{x_i})$ such that
$A_{\lambda}\subset \bigcup 5 B(x_i,r_{x_i})$.
Therefore,
\begin{equation*}
  \begin{split}
    |A_{\lambda}|\lesssim\sum_i|B_i|&\lesssim \left(\sum_i|B_i|^{\frac{n-\alpha}{n}}\right)^{\frac{n}{n-\alpha}}
    \lesssim \left(\sum_i\frac{1}{\lambda}\int_{B_i}d\mu(y)\right)^{\frac{n}{n-\alpha}}\\
    &= \left(\frac{1}{\lambda}\int_{\bigcup_iB_i}d\mu(y)\right)^{\frac{n}{n-\alpha}}
    \lesssim
    \left(\frac{\mu(\mathbb{R}^n)}{\lambda}\right)^{\frac{n}{n-\alpha}}.
  \end{split}
\end{equation*}
This implies that
\begin{equation*}
  \lambda|A_{\lambda}|^{\frac{n-\alpha}{n}}\lesssim \mu(\mathbb{R}^n).
\end{equation*}
By the arbitrary of $\lambda>0$, we get (\ref{for proof, 1}) and then the desired conclusion follows from Theorem \ref{theorem, fractional maximal operator, measure}. Corollary \ref{corollary, fractional maximal operator, measure} is proved. $\hfill\Box$

\medskip

Furthermore, if the measure $V$ is assumed to be absolutely continuous (with respect to Lebesgue measure), we have following corollaries.

\begin{corollary}\label{corollary, maximal operator, function}
  Let $\phi(x)=\Phi(|x|)$ be a radial function, where $\Phi: [0,\infty)\rightarrow [0,\infty)$ is decreasing.
  Suppose that $\phi$ has a continuous integrable radially decreasing majorant.
  Then, for any fixed $\rho>0$, we have
  \begin{equation}
    \lim_{t\rightarrow0^{+}}
    \Big\|\mathcal {M}_{\phi}(V_t)(\cdot)-\sup_{r>0}\phi_r(\cdot)V(\mathbb{R}^n)\Big\|_{L^{1,\infty}(\mathbb{R}^n\backslash B(0,\rho))}=0
  \end{equation}
  for all finite positive absolutely continuous measure $V$.
\end{corollary}
\begin{proof}
  Since $\phi$ has a continuous integrable radially decreasing majorant,
  we have $\mathcal {M}_{\phi}(\mu)(x)\lesssim \mathcal {M}(\mu)(x)$ for all absolutely continuous measure,
  thanks to \cite[Corollary 2.1.12]{Grafakos_Classical_2008}. Then
  \begin{equation*}
    \|\mathcal {M}_{\phi}\mu\|_{L^{1,\infty}}\lesssim \|\mathcal {M}\mu\|_{L^{1,\infty}}\lesssim \mu(\mathbb{R}^n).
  \end{equation*}
  Denote by $K$ the continuous integrable radially decreasing majorant,
  then $\phi(x)\leq K(x), x\in \mathbb{R}^n$.
  Thus,
  \begin{equation*}
    \begin{split}
      \sup_{r>0}\phi_r(e_1)=\sup_{r>0}\frac{1}{r^n}\phi(\frac{e_1}{r})
      \sim
      \int_{B(0,r^{-1})}\phi(\frac{e_1}{r})dx
      \lesssim
      \int_{B(0,r^{-1})}\phi(x)dx\lesssim \|K\|_{L^1}<\infty.
    \end{split}
  \end{equation*}
  The desired conclusion then follows from Theorem \ref{theorem, fractional maximal operator, measure}.
\end{proof}

\begin{corollary}\label{corollary, fractional maximal operator, function}
  Let $\alpha\in (0,n)$, $\phi(x)=\Phi(|x|)$ be a radial function such that $\sup_{r>0}\phi^{\alpha}_r(e_1)<\infty$,
  where $\Phi: [0,\infty)\rightarrow [0,\infty)$ is decreasing, $e_1=(1,0,\cdots,0)$ is the vector on the unit sphere $\mathbb{S}^{n-1}$.
  Then, for any fixed $\rho>0$, we have
  \begin{equation}
    \lim_{t\rightarrow0^{+}}
   \Big\|\mathcal {M}^{\alpha}_{\phi}(V_t)(\cdot)-\sup_{r>0}\phi^{\alpha}_r(\cdot)V(\mathbb{R}^n)\Big\|_{L^{\frac{n}{n-\alpha},\infty}(\mathbb{R}^n\backslash B(0,\rho))}=0
  \end{equation}
  for all finite positive absolutely continuous measure $V$.
\end{corollary}
\begin{proof}
  Since $\sup_{r>0}\phi^{\alpha}_r(e_1)<\infty$, we have
  \begin{equation*}
    \begin{split}
      \phi(x)=\phi(\frac{e_1}{1/|x|})=\frac{1}{|x|^{n-\alpha}}\frac{1}{1/|x|^{n-\alpha}}\phi(\frac{e_1}{1/|x|})
      \lesssim \frac{1}{|x|^{n-\alpha}}\sup_{r>0}\phi^{\alpha}_r(e_1).
    \end{split}
  \end{equation*}
  Then,
  \begin{equation*}
    \begin{split}
      \mathcal {M}^{\alpha}_{\phi}(\mu)(x)
    = &
    \sup_{r>0}\frac{1}{r^{n-\alpha}}\int_{\mathbb{R}^n}\phi(\frac{x-y}{r})d\mu(y)
    \\
    \lesssim &
    \int_{\mathbb{R}^n}\frac{1}{|x-y|^{n-\alpha}}d\mu(y)\sim I_{\alpha}(\mu)(x)
    \end{split}
  \end{equation*}
for all absolutely continuous measure.
  Thus,
  \begin{equation*}
    \|\mathcal {M}^{\alpha}_{\phi}(\mu)(\cdot)\|_{L^{\frac{n}{n-\alpha},\infty}}
    \lesssim
    \|I_{\alpha}(\mu)(\cdot)\|_{L^{\frac{n}{n-\alpha},\infty}}
    \lesssim
    \mu(\mathbb{R}^n).
  \end{equation*}
  This leads to the desired conclusion by Theorem \ref{theorem, fractional maximal operator, measure}.
\end{proof}

Next, we give two specific applications.

\begin{corollary}\label{corollary, maximal operator, possion kernel}
  Let $P(x)=\frac{c_n}{(1+|x|^2)^{\frac{n+1}{2}}}$, where $c_n$ is the constant such that $\int_{\mathbb{R}^n}P(x)dx=1$.
  The function $P$ is called Possion Kernel. We define the maximal function with possion kernel by
  \begin{equation*}
    \mathcal {M}_P(\mu)(x):=\sup_{r>0}P_r\ast \mu(x),
  \end{equation*}
  where $\mu$ is any fixed positive measure.
  Then, for any fixed $\rho>0$, we have
  \begin{equation}
    \lim_{t\rightarrow0^{+}}
   \Big\|\mathcal {M}_{P}(V_t)(\cdot)-\frac{c_nn^{\frac{n}{2}}}{(1+n)^{\frac{n+1}{2}}}V(\mathbb{R}^n)\Big\|_{L^{\frac{n}{n-\alpha},\infty}(\mathbb{R}^n\backslash B(0,\rho))}=0
  \end{equation}
  for all finite positive absolutely continuous measure $V$.
\end{corollary}
\begin{proof}
  Since Possion kernel itself can be treated as a continuous integrable radially decreasing majorant,
  in order to use Theorem \ref{theorem, fractional maximal operator, measure},
  we only need to calculate $\sup_{r>0}P_r(e_1)$.
  In fact
  \begin{equation*}
    \sup_{r>0}P_r(e_1)=\sup_{r>0}\frac{c_n}{r^n(1+r^{-2})^{\frac{n+1}{2}}}=\frac{c_n}{\inf_{r>0}r^n(1+r^{-2})^{\frac{n+1}{2}}}.
  \end{equation*}
A direct calculation yields that
\begin{equation*}
  \begin{split}
    [r^n(1+r^{-2})^{\frac{n+1}{2}}]'
    = &
    nr^{n-1}(1+r^{-2})^{\frac{n+1}{2}}-(n+1)r^{n-3}(1+r^{-2})^{\frac{n-1}{2}}
    \\
    = &
  nr^{n-3}(1+r^{-2})^{\frac{n-1}{2}}(r^2-1/n).
  \end{split}
\end{equation*}
Thus, the function $r^n(1+r^{-2})^{\frac{n+1}{2}}$ takes its minimal value at $r=\frac{1}{\sqrt{n}}$.
\begin{equation*}
  \sup_{r>0}P_r(e_1)=\frac{c_n}{r^n(1+r^{-2})^{\frac{n+1}{2}}}\Bigg|_{r=1/\sqrt{n}}=\frac{c_nn^{\frac{n}{2}}}{(1+n)^{\frac{n+1}{2}}}.
\end{equation*}
\end{proof}

\begin{corollary}\label{corollary, maximal operator, heat kernel}
  The function $G(x)=e^{-\pi |x|^2}$ is called heat kernel. We define the maximal function with heat kernel by
  \begin{equation*}
    \mathcal {M}_G(\mu)(x):=\sup_{r>0}G_r\ast \mu(x),
  \end{equation*}
  where $\mu$ is any fixed positive measure.
  Then, for any fixed $\rho>0$, we have
  \begin{equation}
    \lim_{t\rightarrow0^{+}}
    \Big\|M_{G}(V_t)(\cdot)-\left(\frac{n}{2\pi e}\right)^{\frac{n}{2}}V(\mathbb{R}^n)\Big\|_{L^{\frac{n}{n-\alpha},\infty}(\mathbb{R}^n\backslash B(0,\rho))}=0
  \end{equation}
  for all finite positive absolutely continuous measure $V$.
\end{corollary}
\begin{proof}
  As in the proof of the above corollary, we only need to calculate $\sup_{r>0}G_r(e_1)$.
  Note that
  \begin{equation*}
    \sup_{r>0}G_r(e_1)=\sup_{r>0}\frac{e^{-\pi/r^2}}{r^n}=\frac{1}{\inf_{r>0}r^ne^{\pi/r^2}}.
  \end{equation*}
A direct calculation yields that
\begin{equation*}
  \begin{split}
    [r^ne^{\pi/r^2}]'
    = &
    nr^{n-1}e^{\pi/r^2}+r^n(-2\pi r^{-3})e^{\pi/r^2}
    \\
    = &
    r^{n-3}e^{\pi/r^2}(nr^2-2\pi).
  \end{split}
\end{equation*}
Thus, the function $r^ne^{\pi/r^2}$ takes its minimal value at $r=\frac{\sqrt{2\pi}}{\sqrt{n}}$.
\begin{equation*}
  \sup_{r>0}G_r(e_1)
  =\frac{1}{r^ne^{\pi/r^2}}\Bigg|_{r=\frac{\sqrt{2\pi}}{\sqrt{n}}}
  =\frac{n^{{n}/{2}}}{(2\pi)^{{n}/{2}}e^{{n}/{2}}}
  =\left(\frac{n}{2\pi e}\right)^{{n}/{2}}.
\end{equation*}
\end{proof}

\begin{remark}
  In the proofs of this section, the radial decreasing property is important.
  In fact, we can take a function $\phi$ without radial decreasing property
  such that $\mathcal {M}_{\phi}V_t(x)\rightarrow \sup_{r>0}\phi_r(x)V(\mathbb{R}^n)$ is negative,
  even in the type-3 sense.  Let $\phi(x)=1$ for $|x|=1$, and disappear otherwise.
  Let $dV(x)=\chi_{B(0,1)}(x)dx$, where $dx$ is the Lebesgue measure.
  Note that $\sup_{r>0}\phi_r(x)=\frac{1}{|x|^n}$ and
  \begin{equation*}
    \mathcal {M}_{\phi}V_t(x)=\sup_{r>0}\frac{1}{r^n}\int_{\mathbb{R}^n}\phi(\frac{x-y}{r})dV_t(y)
    =\sup_{r>0}\frac{1}{r^n}\int_{B(0,t)}\phi(\frac{x-y}{r})dy=0.
  \end{equation*}
  Hence, for every fixed $\lambda>0$, $|\{x: \mathcal {M}_{\phi}V_t(x)>\lambda\}|=0$,
  but $|\{x: \sup_{r>0}\phi_r(x)V(\mathbb{R}^n)>\lambda\}|\neq 0$.
  We get the desired conclusion.
\end{remark}

\section{Maximal operator associated with homogeneous functions}

\medskip

This section is concerned with the maximal operator $M_{\Omega}^{\alpha}$, where $\Omega$ is a homogeneous function of degree zero.
Firstly, we list some basic properties of $\Omega$ as follows:
\begin{enumerate}[(A)]
  \item
  $\Big|\Big\{x\in \mathbb{R}^n: \,\frac{|\Omega(x)|}{|x|^{n-\alpha}}>\lambda\Big\}\Big|
  =\lambda^{\frac{-n}{n-\alpha}}\Big|\Big\{x\in \mathbb{R}^n: \frac{|\Omega(x)|}{|x|^{n-\alpha}}>1\Big\}\Big|$;
  \item
  $\|\Omega\|^{\frac{n}{n-\alpha}}_{L^{\frac{n}{n-\alpha}}(\mathbb{S}^{n-1})}
  =n\lambda^{\frac{n}{n-\alpha}}\Big|\Big\{x\in \mathbb{R}^n: \frac{|\Omega(x)|}{|x|^{n-\alpha}}>\lambda\Big\}\Big|.$
\end{enumerate}

\medskip

{\it Proof of Theorem \ref{Theorem, fractional maximal with homogeneous kernel}}.\quad
Without loss of generality, we assume that $V$ is a probability measure. For $t>0$, let $V_t^1,\,V_t^2$, $r_t$, $\epsilon_t$ be as in the proof of Theorem 1.1.
For $\lambda>0$, denote
\begin{equation*}
  E_{t}^{\lambda}:=\{x\in \mathbb{R}^n: M^{\alpha}_{\Omega}V_t(x)>\lambda\},
\end{equation*}
and
\begin{equation*}
  E_{t}^{\lambda,1}:=\{x\in \mathbb{R}^n: M^{\alpha}_{\Omega}V^1_t(x)>\lambda\},\hspace{6mm}E_{t}^{\lambda,2}:=\{x\in \mathbb{R}^n: M^{\alpha}_{\Omega}V^2_t(x)>\lambda\}.
\end{equation*}
For fixed $\nu>0$, recalling that $M^{\alpha}_{\Omega}$ is boundedness from $L^1$ to $L^{\frac{n}{n-\alpha},\infty}$, we obtain that
\begin{equation*}
  \nu\lambda|E_{t}^{\nu\lambda,2}|^{\frac{n-\alpha}{n}}\lesssim V^2_t(\mathbb{R}^n)= \epsilon_t \rightarrow 0^+\ \text{as}\ t \rightarrow 0^+.
\end{equation*}
For fixed $\rho>2r_t$, $x\in B^c(0,\rho)$,
\begin{equation*}
  \begin{split}
    M^{\alpha}_{\Omega}V^1_t(x)
        = &
    \sup_{r>0}\frac{1}{r^{n-\alpha}}\int_{B(x,r)\cap B(0,r_t)}|\Omega(x-y)|dV^1_t(y)
    \\
    = &
    \sup_{|x|-r_t\leq r\leq|x|+r_t}\frac{1}{r^{n-\alpha}}\int_{B(x,r)\cap B(0,r_t)}|\Omega(x-y)|dV^1_t(y)
    \\
    \geq &
    \frac{1}{(|x|+r_t)^{n-\alpha}}\int_{B(0,r_t)}|\Omega(x-y)|dV^1_t(y)
    \\
    \geq &
    \frac{1}{(|x|+r_t)^{n-\alpha}}\int_{B(0,r_t)}|\Omega(x)|dV^1_t(y)
    -
    \frac{1}{(|x|+r_t)^{n-\alpha}}\int_{B(0,r_t)}|\Omega(x-y)-\Omega(x)|dV^1_t(y)
    \\
    = &
    \frac{(1-\epsilon_t)|\Omega(x)|}{(|x|+r_t)^{n-\alpha}}
    -
    \frac{1}{(|x|+r_t)^{n-\alpha}}\int_{B(0,r_t)}|\Omega(x-y)-\Omega(x)|dV^1_t(y)
    \\
    \geq &
    \frac{(1-\epsilon_t)\rho^{n-\alpha}}{(\rho+r_t)^{n-\alpha}}\cdot\frac{|\Omega(x)|}{|x|^{n-\alpha}}
    -
    \frac{1}{|x|^{n-\alpha}}\int_{B(0,r_t)}|\Omega(x-y)-\Omega(x)|dV^1_t(y).
  \end{split}
\end{equation*}
This implies that
\begin{equation}\label{Theorem, fractional maximal with homogeneous kernel, for proof 1}
  \begin{split}
    M^{\alpha}_{\Omega}V^1_t(x)-\frac{|\Omega(x)|}{|x|^{n-\alpha}}
     \geq
    \left(\frac{(1-\epsilon_t)\rho^{n-\alpha}}{(\rho+r_t)^{n-\alpha}}-1\right)\frac{|\Omega(x)|}{|x|^{n-\alpha}}
    -    \frac{1}{|x|^{n-\alpha}}\int_{B(0,r_t)}|\Omega(x-y)-\Omega(x)|dV^1_t(y).
  \end{split}
\end{equation}
Set
\begin{equation*}
  F_{t,\rho}^{\lambda}:=\Big\{x\in B^c(0,\rho): \frac{(1-\epsilon_t)\rho^{n-\alpha}}{(\rho+r_t)^{n-\alpha}}\cdot\frac{|\Omega(x)|}{|x|^{n-\alpha}}>\lambda\Big\}
\end{equation*}
and
\begin{equation*}
  G_{t,\rho}^{\lambda,1}:=\Big\{x\in B^c(0,\rho): \frac{1}{|x|^{n-\alpha}}\int_{B(0,r_t)}|\Omega(x-y)-\Omega(x)|dV^1_t(y)>\lambda\Big\}.
\end{equation*}
Now, we estimate $|G_{t,\rho}^{\nu\lambda,1}|$ for fixed $\nu>0$.
Using Chebychev's inequality, we conclude that
\begin{equation*}
  \begin{split}
    |G_{t,\rho}^{\nu\lambda,1}|
    \leq &
    \frac{1}{\nu\lambda}\int_{B^c(0,\rho)}\frac{1}{|x|^{n-\alpha}}\int_{B(0,r_t)}|\Omega(x-y)-\Omega(x)|dV^1_t(y)dx
    \\
    = &
    \frac{1}{\nu\lambda}\int_{B(0,r_t)}\int_{B^c(0,\rho)}\frac{|\Omega(x-y)-\Omega(x)|}{|x|^{n-\alpha}}dx dV^1_t(y)
    \\
    \leq &
    \frac{1}{\nu\lambda}\sup_{y\in B(0,r_t)}\int_{B^c(0,\rho)}\frac{|\Omega(x-y)-\Omega(x)|}{|x|^{n-\alpha}}dx,
  \end{split}
\end{equation*}
where
\begin{equation*}
  \begin{split}
    \int_{B^c(0,\rho)}\frac{|\Omega(x-y)-\Omega(x)|}{|x|^{n-\alpha}}dx
        = &
    \int_{\rho}^{\infty}\int_{\mathbb{S}^{n-1}}\frac{|\Omega(x'-y/r)-\Omega(x')|}{r^{n-\alpha}}d\sigma(x')r^{n-1}dr
    \\
    \leq &
    \int_{\rho}^{\infty}\frac{\omega(|y|/r)}{r^{1-\alpha}}dr
    =
    |y|^{\alpha}\cdot \int_{0}^{|y|/\rho}\frac{\omega(s)}{s^{1+\alpha}}ds.
  \end{split}
\end{equation*}
Thus,
\begin{equation*}
  \begin{split}
    |G_{t,\rho}^{\nu\lambda,1}|
    \leq &
    \frac{1}{\nu\lambda}\sup_{y\in B(0,r_t)}\int_{B^c(0,\rho)}\frac{|\Omega(x-y)-\Omega(x)|}{|x|^{n-\alpha}}dx
    \\
    \leq &
    \frac{1}{\nu\lambda}\sup_{y\in B(0,r_t)}\left(|y|^{\alpha}\cdot \int_{0}^{|y|/\rho}\frac{\omega(s)}{s^{1+\alpha}}ds\right)\rightarrow 0
  \end{split}
\end{equation*}
as $t\rightarrow 0^+$.

Observing $E_{t}^{\lambda} \supset E_{t}^{(1+\nu)\lambda,1}\backslash E_{t}^{\nu\lambda,2}$
and $E_{t}^{(1+\nu)\lambda,1}\supset F_{t,\rho}^{(1+2\nu)\lambda}\backslash G_{t,\rho}^{\nu\lambda,1}$, we deduce that
\begin{equation*}
  \begin{split}
    |E_{t}^{\lambda}|
    \geq
    |E_{t}^{(1+\nu)\lambda,1}|-|E_{t}^{\nu\lambda,2}|
        \geq
    |F_{t,\rho}^{(1+2\nu)\lambda}|-|G_{t,\rho}^{\nu,1}|-|E_{t}^{\nu\lambda,2}|.
  \end{split}
\end{equation*}
Noting that $|E_{t}^{\nu\lambda,2}|$, $|G_{t,\rho}^{\nu,1}|\rightarrow 0$ as $t\rightarrow 0^+$, we have
\begin{equation*}
  \begin{split}
    \varliminf_{t\rightarrow 0^{+}}|E_{t}^{\lambda}|
    \geq &
    \varliminf_{t\rightarrow 0^{+}}|F_{t,\rho}^{(1+2\nu)\lambda}|
    \\
    \geq &
    \varliminf_{t\rightarrow 0^{+}}\Big|\Big\{x\in \mathbb{R}^n: \frac{|\Omega(x)|}{|x|^{n-\alpha}}>\frac{\lambda(\rho+r_t)^{n-\alpha}(1+2\nu)}{\rho^{n-\alpha}(1-\epsilon_t)}\Big\}\Big|-|B(0,\rho)|
    \\
    = &
    \varliminf_{t\rightarrow 0^{+}}\left(\frac{(\rho+r_t)^{n-\alpha}(1+2\nu)}{\rho^{n-\alpha}(1-\epsilon_t)}\right)^{\frac{-n}{n-\alpha}}
    \Big|\Big\{x\in \mathbb{R}^n: \frac{|\Omega(x)|}{|x|^{n-\alpha}}>\lambda\Big\}\Big|-|B(0,\rho)|
    \\
    = &
    (1+2\nu)^{\frac{-n}{n-\alpha}}\Big|\Big\{x\in \mathbb{R}^n: \frac{|\Omega(x)|}{|x|^{n-\alpha}}>\lambda\Big\}\Big|-|B(0,\rho)|,
  \end{split}
\end{equation*}
where we use Property (A). Letting $\nu\rightarrow 0$ and $\rho\rightarrow 0$, we obtain
\begin{equation}\label{proposition for homogeneous maximal, measure inequality for lower bound}
  \varliminf_{t\rightarrow 0^{+}}|E_{t}^{\lambda}|\geq \Big|\Big\{x\in \mathbb{R}^n: \frac{|\Omega(x)|}{|x|^{n-\alpha}}>\lambda\Big\}\Big|.
\end{equation}
Recalling the definition of $E_t^{\lambda}$ and the boundedness of $M_{\Omega}^{\alpha}$, we obtain
\begin{equation*}
  \begin{split}
    \lambda\Big|\Big\{x\in \mathbb{R}^n: \frac{\Omega(x)}{|x|^{n-\alpha}}>\lambda\Big\}\Big|^{\frac{n-\alpha}{n}}
    \leq
    \varliminf_{t\rightarrow 0^{+}}\lambda|E_{t}^{\lambda}|^{\frac{n-\alpha}{n}}
        \leq
    \varliminf_{t\rightarrow 0^{+}}\|M_{\Omega}^{\alpha}V_t\|_{L^{\frac{n}{n-\alpha},\infty}}
        \leq
    \|M_{\Omega}^{\alpha}\|_{L^1\rightarrow L^{\frac{n}{n-\alpha}},\infty}.
  \end{split}
\end{equation*}
By the arbitrary of $\lambda$, we actually have
$$\frac{\Omega(x)}{|x|^{n-\alpha}}\in L^{\frac{n}{n-\alpha},\infty},\quad
{\rm and}\quad \Big\|\frac{\Omega(\cdot)}{|\cdot|^{n-\alpha}}\Big\|_{L^{\frac{n}{n-\alpha},\infty}}\leq \|M_{\Omega}^{\alpha}\|_{L^1\rightarrow L^{\frac{n}{n-\alpha},\infty}}.$$ This completes the proof of conclusion (1).

On the other hand, by Property (B), we have
\begin{equation*}
  \|\Omega\|^{\frac{n}{n-\alpha}}_{L^{\frac{n}{n-\alpha}}(\mathbb{S}^{n-1})}
  =n\lambda^{\frac{n}{n-\alpha}}\Big|\Big\{x\in \mathbb{R}^n:\, \frac{|\Omega(x)|}{|x|^{n-\alpha}}>\lambda\Big\}\Big|,
\end{equation*}
and then $\Omega \in L^{\frac{n}{n-\alpha}}(\mathbb{S}^{n-1})$,
$\|\Omega \|_{L^{\frac{n}{n-\alpha}}(\mathbb{S}^{n-1})}\lesssim \|M_{\Omega}^{\alpha}\|_{L^1\rightarrow L^{\frac{n}{n-\alpha},\infty}}$. The conclusion (2) is proved.

Next, we turn to verify conclusion (3). For $\rho>2r_t$, $x\in B^c(0,\rho)$, we have
\begin{equation*}
  \begin{split}
    M^{\alpha}_{\Omega}V^1_t(x)
    = &
    \sup_{r>0}\frac{1}{r^{n-\alpha}}\int_{B(x,r)\cap B(0,r_t)}|\Omega(x-y)|dV^1_t(y)
    \\
    = &
    \sup_{|x|-r_t\leq r\leq|x|+r_t}\frac{1}{r^{n-\alpha}}\int_{B(x,r)\cap B(0,r_t)}|\Omega(x-y)|dV^1_t(y)
    \\
    \leq &
    \frac{1}{(|x|-r_t)^{n-\alpha}}\int_{B(0,r_t)}|\Omega(x-y)|dV^1_t(y)
    \\
    \leq &
    \frac{1}{(|x|-r_t)^{n-\alpha}}\int_{B(0,r_t)}|\Omega(x-y)-\Omega(x)|dV^1_t(y)
    +
    \frac{1}{(|x|-r_t)^{n-\alpha}}\int_{B(0,r_t)}|\Omega(x)|dV_t(y)
    \\
    \leq &
    \frac{1}{(|x|-r_t)^{n-\alpha}}\int_{B(0,r_t)}|\Omega(x-y)-\Omega(x)|dV^1_t(y)
    +
    \frac{|\Omega(x)|}{(|x|-r_t)^{n-\alpha}}
    \\
    \leq &
    \frac{2}{|x|^{n-\alpha}}\int_{B(0,r_t)}|\Omega(x-y)-\Omega(x)|dV^1_t(y)
    +
    \frac{\rho^{n-\alpha}}{(\rho-r_t)^{n-\alpha}}\cdot\frac{|\Omega(x)|}{|x|^{n-\alpha}}.
  \end{split}
\end{equation*}
This implies that for $x\in B(0,\rho)^c$,
\begin{equation*}\label{Theorem, fractional maximal with homogeneous kernel, for proof 2}
  \begin{split}
    M^{\alpha}_{\Omega}V^1_t(x)-\frac{|\Omega(x)|}{|x|^{n-\alpha}}
     \leq
    \left(\frac{\rho^{n-\alpha}}{(\rho-r_t)^{n-\alpha}}-1\right)\frac{|\Omega(x)|}{|x|^{n-\alpha}}
    +
    \frac{2}{|x|^{n-\alpha}}\int_{B(0,r_t)}|\Omega(x-y)-\Omega(x)|dV^1_t(y).
  \end{split}
\end{equation*}
The above inequality together with  (\ref{Theorem, fractional maximal with homogeneous kernel, for proof 1}) yields that
there exists a sequence $\beta_t\rightarrow 0^+$ as $t\rightarrow 0^+$, such that
\begin{equation*}
  \begin{split}
    \left|M^{\alpha}_{\Omega}V^1_t(x)-\frac{|\Omega(x)|}{|x|^{n-\alpha}}\right|
      \leq
    \beta_t\frac{|\Omega(x)|}{|x|^{n-\alpha}}
    +
    \frac{2}{|x|^{n-\alpha}}\int_{B(0,r_t)}|\Omega(x-y)-\Omega(x)|dV^1_t(y).
  \end{split}
\end{equation*}
Denote
\begin{equation*}
  A_{t}^{\lambda}:=\Big\{x\in \mathbb{R}^n:\, \Big|M^{\alpha}_{\Omega}V_t(x)-\frac{|\Omega(x)|}{|x|^{n-\alpha}}\Big|>\lambda\Big\},
\end{equation*}
\begin{equation*}
  A_{t}^{\lambda,1}:=\Big\{x\in \mathbb{R}^n:\, \Big|M^{\alpha}_{\Omega}V^1_t(x)-\frac{|\Omega(x)|}{|x|^{n-\alpha}}\Big|>\lambda\Big\},
\end{equation*}
and $A_{t,\rho}^{\lambda,1}=F_{t}^{\lambda,1}\cap B^c(0,\rho)$.
Then
\begin{equation*}
  \begin{split}
    A_t^{\lambda}
    \subset
    A_{t}^{(1-\nu)\lambda,1}\cup E_t^{\nu\lambda,2}
    \subset
    B(0,\rho)\cup A_{t,\rho}^{(1-\nu)\lambda,1}\cup E_t^{\nu\lambda,2},
  \end{split}
\end{equation*}
which implies that $|A_t^{\lambda}|\leq |B(0,\rho)|+|A_{t,\rho}^{(1-\nu)\lambda,1}|+|E_t^{\nu\lambda,2}|$.
A direct calculation yields that
\begin{equation*}
  \begin{split}
    |A_{t,\rho}^{(1-\nu)\lambda,1}|
      &\leq
    \Big|\Big\{x\in B^c(0,\rho): \beta_t\frac{|\Omega(x)|}{|x|^{n-\alpha}}
    +
    \frac{2}{|x|^{n-\alpha}}\int_{B(0,r_t)}|\Omega(x-y)-\Omega(x)|dV^1_t(y)>(1-\nu)\lambda\Big\}\Big|
    \\
    &\leq
    \Big|\Big\{x\in \mathbb{R}^n: \beta_t\frac{|\Omega(x)|}{|x|^{n-\alpha}}>(1-\nu)\lambda/2\Big\}\Big|
    \\
    &\qquad+
    \Big|\Big\{x\in B^c(0,\rho): \frac{2}{|x|^{n-\alpha}}\int_{B(0,r_t)}|\Omega(x-y)-\Omega(x)|dV^1_t(y)>(1-\nu)\lambda/2\Big\}\Big|
    \\
    &\leq
    \left(\frac{2}{(1-\nu)\lambda}\beta_t\Big\|\frac{\Omega(\cdot)}{|\cdot|^{n-\alpha}}\Big\|_{L^{\frac{n}{n-\alpha},\infty}}\right)^{\frac{n}{n-\alpha}}
    +|G_{t,\rho}^{(1-\nu)\lambda/4,1}|
    \rightarrow 0\ \text{as}\ t\rightarrow 0^+.
  \end{split}
\end{equation*}
Recalling $E_t^{\nu\lambda,2}\rightarrow 0$ as $t\rightarrow 0^+$, we deduce that
\begin{equation}
  \varlimsup_{t\rightarrow 0^+}|A_t^{\lambda}|
  \leq
  |B(0,\rho)|+\varlimsup_{t\rightarrow 0^+}|A_{t,\rho}^{(1-\nu)\lambda,1}|+\varlimsup_{t\rightarrow 0^+}|E_t^{\nu\lambda,2}|\leq |B(0,\rho)|,
\end{equation}
which yields the desired conclusion by letting $\rho \rightarrow 0$.$\hfill\Box$

\medskip

{\it Proof of Theorem \ref{Theorem, fractional maximal with homogeneous kernel, Lq-Dini}}.\quad
  Since $\Omega$ satisfies the $L^{\frac{n}{n-\alpha}}$-Dini condition, we have $\Omega\in L^{\frac{n}{n-\alpha}}(\mathbb{S}^{n-1})$.
By Property (B), we conclude that
$\frac{\Omega(x)}{|x|^{n-\alpha}}\in L^{\frac{n}{n-\alpha},\infty}$.

As in the proof of Theorem \ref{Theorem, fractional maximal with homogeneous kernel}, denote
\begin{equation*}
  G_{t,\rho}^{\lambda,1}:=\Big\{x\in B^c(0,\rho):\, \frac{1}{|x|^{n-\alpha}}\int_{B(0,r_t)}|\Omega(x-y)-\Omega(x)|dV^1_t(y)>\lambda\Big\}.
\end{equation*}
By Minkowski's inequality and the embedding $L^{\frac{n}{n-\alpha}}\subset L^{\frac{n}{n-\alpha},\infty}$, we conclude that
\begin{equation*}
  \begin{split}
    \sup_{\lambda>0}\lambda|G_{t,\rho}^{\lambda,1}|^{\frac{n-\alpha}{n}}
    = &
    \Big\|\frac{1}{|\cdot|^{n-\alpha}}\int_{B(0,r_t)}|\Omega(\cdot-y)-\Omega(\cdot)|dV^1_t(y)\Big\|_{L^{\frac{n}{n-\alpha}}(\mathbb{R}^n\backslash B(0,\rho))}
    \\
    \leq &
    \int_{B(0,r_t)}
    \left(\int_{B^c(0,\rho)}\frac{|\Omega(x-y)-\Omega(x)|^{\frac{n}{n-\alpha}}}{|x|^n}dx\right)^{\frac{n-\alpha}{n}} dV^1_t(y)
    \\
    \leq &
    \sup_{y\in B(0,r_t)}\left(\int_{B^c(0,\rho)}\frac{|\Omega(x-y)-\Omega(x)|^{\frac{n}{n-\alpha}}}{|x|^n}dx\right)^{\frac{n-\alpha}{n}},
  \end{split}
\end{equation*}
where
\begin{equation*}
  \begin{split}
    \int_{B^c(0,\rho)}\frac{|\Omega(x-y)-\Omega(x)|^{\frac{n}{n-\alpha}}}{|x|^n}dx
        = &
    \int_{\rho}^{\infty}\int_{\mathbb{S}^{n-1}}\frac{|\Omega(x'-y/r)-\Omega(x')|^{\frac{n}{n-\alpha}}}{r^n}d\sigma(x')r^{n-1}dr
    \\
    \leq &
    \int_{\rho}^{\infty}\frac{\omega_{\frac{n}{n-\alpha}}(|y|/r)^{\frac{n}{n-\alpha}}}{r}dr
    =
    \int_{0}^{|y|/\rho}\frac{\omega_{\frac{n}{n-\alpha}}(s)^{\frac{n}{n-\alpha}}}{s}ds.
  \end{split}
\end{equation*}
Since $\Omega$ satisfies the $L^{\frac{n}{n-\alpha}}$-Dini condition, we have
\begin{equation*}
  \int_{0}^{1}\frac{\omega_{\frac{n}{n-\alpha}}(s)}{s}ds<\infty,
\end{equation*}
which implies that $\omega_{\frac{n}{n-\alpha}}(s)\rightarrow 0$ as $s\rightarrow 0^+$.
Thus,
\begin{equation*}
  \begin{split}
    \sup_{\lambda>0}\lambda|G_{t,\rho}^{\lambda,1}|^{\frac{n-\alpha}{n}}
    \leq &
    \sup_{y\in B(0,r_t)}\left(\int_{B^c(0,\rho)}\frac{|\Omega(x-y)-\Omega(x)|^{\frac{n}{n-\alpha}}}{|x|}dx\right)^{\frac{n-\alpha}{n}}
    \\
    \leq &
    \sup_{y\in B(0,r_t)}\left(\int_{0}^{|y|/\rho}\frac{\omega_{\frac{n}{n-\alpha}}(s)^{\frac{n}{n-\alpha}}}{s}ds\right)^{\frac{n-\alpha}{n}}
        \\
    \leq &
    \sup_{y\in B(0,r_t)}\omega_{\frac{n}{n-\alpha}}(s)^{\frac{\alpha}{n}}
    \cdot
    \sup_{y\in B(0,r_t)}\left(\int_{0}^{|y|/\rho}\frac{\omega_{\frac{n}{n-\alpha}}(s)}{s}ds\right)^{\frac{n-\alpha}{n}}
    \rightarrow 0
  \end{split}
\end{equation*}
as $t\rightarrow 0^+$. So we have
\begin{equation*}
  \begin{split}
    &\Big\|\frac{1}{|\cdot|^{n-\alpha}}\int_{B(0,r_t)}|\Omega(\cdot-y)-\Omega(\cdot)|dV^1_t(y)\Big\|_{L^{\frac{n}{n-\alpha},\infty}(\mathbb{R}^n\backslash B(0,\rho))}
    \\
   &\qquad\qquad\qquad =
    \sup_{\lambda>0}\lambda|G_{t,\rho}^{\lambda,1}|^{\frac{n-\alpha}{n}}\rightarrow 0\ \text{as}\ t\rightarrow 0^+.
  \end{split}
\end{equation*}
Using the same method as in the proof of Theorem \ref{Theorem, fractional maximal with homogeneous kernel}, there exist $\beta_t\rightarrow 0^+$ as $t\rightarrow 0^+$, such that
\begin{equation*}
  \begin{split}
    \left|M^{\alpha}_{\Omega}V^1_t(x)-\frac{|\Omega(x)|}{|x|^{n-\alpha}}\right|
        \leq
    \beta_t\frac{|\Omega(x)|}{|x|^{n-\alpha}}
    +
    \frac{2}{|x|^{n-\alpha}}\int_{B(0,r_t)}|\Omega(x-y)-\Omega(x)|dV^1_t(y).
  \end{split}
\end{equation*}
Then,
\begin{equation*}
  \begin{split}
    &\Big\|M^{\alpha}_{\Omega}V^1_t(\cdot)-\frac{|\Omega(\cdot)|}{|\cdot|^{n-\alpha}}\Big\|_{L^{\frac{n}{n-\alpha},\infty}(\mathbb{R}^n\backslash B(0,\rho))}
        \lesssim
    \beta_t\Big\|\frac{|\Omega(\cdot)|}{|\cdot|^{n-\alpha}}\Big\|_{L^{\frac{n}{n-\alpha},\infty}(\mathbb{R}^n)}
    \\
    &\qquad\qquad+
    \Big\|\frac{2}{|\cdot|^{n-\alpha}}\int_{B(0,r_t)}|\Omega(\cdot-y)-\Omega(\cdot)|dV^1_t(y)\Big\|_{L^{\frac{n}{n-\alpha},\infty}(\mathbb{R}^n\backslash B(0,\rho))}\rightarrow 0\ \text{as}\ t\rightarrow 0^+.
  \end{split}
\end{equation*}
Recalling
\begin{equation*}
  \|M^{\alpha}_{\Omega}V^2_t(\cdot)\|_{L^{\frac{n}{n-\alpha},\infty}(\mathbb{R}^n)}\lesssim V_t^2(\mathbb{R}^n)=\epsilon_t\rightarrow 0\ \text{as}\ t\rightarrow 0^+,
\end{equation*}
we conclude that
\begin{equation*}
  \begin{split}
    \varlimsup_{t\rightarrow 0^+}
    \Big\|M^{\alpha}_{\Omega}V_t(\cdot)-\frac{|\Omega(\cdot)|}{|\cdot|^{n-\alpha}}\Big\|_{L^{\frac{n}{n-\alpha},\infty}(\mathbb{R}^n\backslash B(0,\rho))}
       &\lesssim
    \lim_{t\rightarrow 0^+}
    \Big\|M^{\alpha}_{\Omega}V^1_t(\cdot)-\frac{|\Omega(\cdot)}{|\cdot|^{n-\alpha}}\Big\|_{L^{\frac{n}{n-\alpha},\infty}(\mathbb{R}^n\backslash B(0,\rho))}
    \\
    &\qquad\qquad +
    \lim_{t\rightarrow 0^+}\Big\|M^{\alpha}_{\Omega}V^2_t(\cdot)\Big\|_{L^{\frac{n}{n-\alpha},\infty}(\mathbb{R}^n)}=0.
  \end{split}
\end{equation*}
This implies that (3) holds.

Next, observe that for any fixed $\lambda >0$, $|G_{t,\rho}^{\lambda,1}|\rightarrow 0$ as $t\rightarrow 0$.
By the same aeguments used in the proof of Theorem \ref{Theorem, fractional maximal with homogeneous kernel}, we can verify that
\begin{equation*}
  \|\Omega \|_{L^{\frac{n}{n-\alpha}}(\mathbb{S}^{n-1})}
  \sim
  \|\frac{\Omega(x)}{|x|^{n-\alpha}}\|_{L^{\frac{n}{n-\alpha},\infty}}\lesssim \|M_{\Omega}^{\alpha}\|_{L^1\rightarrow L^{\frac{n}{n-\alpha},\infty}}.
\end{equation*}
This completes the proof of conclusion (1) and (2).
Theorem \ref{Theorem, fractional maximal with homogeneous kernel, Lq-Dini} is proved.

\section{Limiting weak-type behaviors of the singular and fractional integral operators}

This section is devoted to the proofs of the limiting weak-type behaviors of $T_{\Omega}^\alpha$.

\medskip

{\it Proof of Theorem \ref{Theorem, fractional operator with homogeneous kernel}.}\quad
  Without loss of generality, we assume that $V$ is a probability measure. For $t>0$, let $V_t^1,\,V_t^2$, $r_t,\,\epsilon_t$ be as in the proof of Theorem 1.1.
  For $\rho>2r_t$, $x\in B^c(0,\rho)$, we have
\begin{equation}\label{for proof, 5}
  \begin{split}
    \left|T^{\alpha}_{\Omega}V^1_t(x)-\frac{\Omega(x)}{|x|^{n-\alpha}}\right|
    = &
    \left|\int_{\mathbb{R}^n}\Big(\frac{\Omega(x-y)}{|x-y|^{n-\alpha}}-\frac{\Omega(x)}{|x|^{n-\alpha}}\Big)dV^1_t(y)\right|
    +\frac{\epsilon_t|\Omega(x)|}{|x|^{n-\alpha}}
    \\
    \leq &
    \left|\int_{\mathbb{R}^n}\frac{\Omega(x-y)-\Omega(x)}{|x-y|^{n-\alpha}}dV^1_t(y)\right|
    +\left|\int_{\mathbb{R}^n}\Big(\frac{\Omega(x)}{|x-y|^{n-\alpha}}-\frac{\Omega(x)}{|x|^{n-\alpha}}\Big)dV^1_t(y)\right|
    +\frac{\epsilon_t|\Omega(x)|}{|x|^{n-\alpha}}
    \\
    \leq &
    \frac{2^{n-\alpha}}{|x|^{n-\alpha}}\left|\int_{\mathbb{R}^n}|\Omega(x-y)-\Omega(x)|dV^1_t(y)\right|
    +\frac{\beta_t|\Omega(x)|}{|x|^{n-\alpha}},
  \end{split}
\end{equation}
where $\beta_t\rightarrow 0$ as $t\rightarrow 0^+$.
Thus, for $x\in B^c(0,\rho)$,
\begin{equation*}
  \begin{split}
    \left|\frac{\Omega(x)}{|x|^{n-\alpha}}\right|
    \leq &
    \left|T^{\alpha}_{\Omega}V^1_t(x)-\frac{\Omega(x)}{|x|^{n-\alpha}}\right|+|T^{\alpha}_{\Omega}V^1_t(x)|
    \\
    \leq &
    \frac{2^{n-\alpha}}{|x|^{n-\alpha}}\left|\int_{\mathbb{R}^n}|\Omega(x-y)-\Omega(x)|dV^1_t(y)\right|
    +\frac{\beta_t|\Omega(x)|}{|x|^{n-\alpha}}+|T^{\alpha}_{\Omega}V_t^1(x)|.
  \end{split}
\end{equation*}
This implies that
\begin{equation*}
  \frac{(1-\beta_t)|\Omega(x)|}{|x|^{n-\alpha}}
  \leq
  \frac{2}{|x|^{n-\alpha}}\left|\int_{\mathbb{R}^n}|\Omega(x-y)-\Omega(x)|dV^1_t(y)\right|
  +|T^{\alpha}_{\Omega}V_t^1(x)|.
\end{equation*}
Set
\begin{equation*}
  \tilde{E}_{t}^{\lambda}:=\{x\in \mathbb{R}^n: T^{\alpha}_{\Omega}V_t(x)>\lambda\},
\end{equation*}
\begin{equation*}
  \tilde{E}_{t}^{\lambda,1}:=\{x\in \mathbb{R}^n: T^{\alpha}_{\Omega}V^1_t(x)>\lambda\},\hspace{6mm}\tilde{E}_{t}^{\lambda,2}:=\{x\in \mathbb{R}^n: T^{\alpha}_{\Omega}V^2_t(x)>\lambda\},
\end{equation*}
\begin{equation*}
  \tilde{F}_{t,\rho}^{\lambda}=\Big\{x\in B^c(0,\rho): \frac{(1-\beta_t)|\Omega(x)|}{|x|^{n-\alpha}}>\lambda\Big\},
\end{equation*}
\begin{equation*}
\tilde{G}_{t,\rho}^{\lambda,1}:=\Big\{x\in B^c(0,\rho): \frac{1}{|x|^{n-\alpha}}\int_{B(0,r_t)}|\Omega(x-y)-\Omega(x)|dV^1_t(y)>\lambda\Big\}.
\end{equation*}
Then
\begin{equation*}
  \tilde{F}_{t,\rho}^{\lambda}\subset \tilde{G}_{t,\rho}^{\lambda/2^{n+1-\alpha},1}\cup \tilde{E}_{t}^{\lambda/2,1}.
\end{equation*}
Consequently,
\begin{equation}\label{for proof, 6}
  \begin{split}
  \Big|\Big\{x\in \mathbb{R}^n: \frac{(1-\beta_t)|\Omega(x)|}{|x|^{n-\alpha}}>\lambda\Big\}\Big|
  \leq &
  |B(0,\rho)|+|\tilde{F}_{t,\rho}^{\lambda}|
  \\
  \leq &
  |B(0,\rho)|+|\tilde{G}_{t,\rho}^{\lambda/2^{n+1-\alpha},1}|+|\tilde{E}_{t}^{\lambda/2}|
  \\
  \leq &
  |B(0,\rho)|+|\tilde{G}_{t,\rho}^{\lambda/2^{n+1-\alpha},1}|+(2/\lambda)^{\frac{n}{n-\alpha}}\|T_{\Omega}^{\alpha}\|_{L^1\rightarrow L^{\frac{n}{n-\alpha},\infty}},
  \end{split}
\end{equation}
where in the last inequality we use the fact
$$\frac{\lambda}{2}|\tilde{E}_{t}^{\lambda/2,1}|^{\frac{n-\alpha}{n}}\leq \|T_{\Omega}^{\alpha}\|_{L^1\rightarrow L^{\frac{n}{n-\alpha},\infty}}V_t^1(\mathbb{R}^n)\le \|T_{\Omega}^{\alpha}\|_{L^1\rightarrow L^{\frac{n}{n-\alpha},\infty}}.$$
Recall $|\tilde{G}_{t,\rho}^{\lambda/2^{n+1-\alpha},1}|\rightarrow 0$
as $t\rightarrow 0^+$ (see the proof of Theorem \ref{Theorem, fractional maximal with homogeneous kernel}).
Letting $t \rightarrow \infty$ in (\ref{for proof, 6}), and then letting $\rho\rightarrow 0$, we get
\begin{equation*}
  \Big|\Big\{x\in \mathbb{R}^n: \frac{|\Omega(x)|}{|x|^{n-\alpha}}>\lambda\Big\}\Big|
  \leq
  (2/\lambda)^{\frac{n}{n-\alpha}}\|T_{\Omega}^{\alpha}\|_{L^1\rightarrow L^{\frac{n}{n-\alpha},\infty}},
\end{equation*}
which yields the desired conclusion (1). Then, conclusion (2) follows immediately from property (B)
mentioned in Section 3.

Now, we turn to the proof of conclusion (3).
For a fixed $\nu>0$, recalling that $T^{\alpha}_{\Omega}$ is boundedness from $L^1$ to $L^{\frac{n}{n-\alpha},\infty}$, we obtain that
\begin{equation*}
  \nu\lambda|\tilde{E}_{t}^{\nu\lambda,2}|^{\frac{n}{n-\alpha}}\lesssim V^2_t(\mathbb{R}^n)= \epsilon_t \rightarrow 0^+\ \text{as}\ t \rightarrow 0^+.
\end{equation*}

Set
\begin{equation*}
  \tilde{A}_{t}^{\lambda}:=\Big\{x\in \mathbb{R}^n:\, |T^{\alpha}_{\Omega}V_t(x)-\frac{|\Omega(x)|}{|x|^{n-\alpha}}|>\lambda\Big\},
\end{equation*}
\begin{equation*}
  \tilde{A}_{t}^{\lambda,1}:=\Big\{x\in \mathbb{R}^n: |T^{\alpha}_{\Omega}V^1_t(x)-\frac{|\Omega(x)|}{|x|^{n-\alpha}}|>\lambda\Big\},
\end{equation*}
and $\tilde{A}_{t,\rho}^{\lambda,1}:=\tilde{A}_{t}^{\lambda,1}\cap B^c(0,\rho)$.
Observe that
\begin{equation*}
  \begin{split}
    \tilde{A}_t^{\lambda}     \subset
    \tilde{A}_{t}^{(1-\nu)\lambda,1}\cup \tilde{E}_t^{\nu\lambda,2}
    \subset
    B(0,\rho)\cup \tilde{A}_{t,\rho}^{(1-\nu)\lambda,1}\cup \tilde{E}_t^{\nu\lambda,2}.
  \end{split}
\end{equation*}
Recalling $\tilde{G}_{t,\rho}^{\lambda,1}\rightarrow 0$ as $t\rightarrow 0^+$,
and recalling $\frac{\Omega(x)}{|x|^{\frac{n}{n-\alpha}}}\in L^{\frac{n}{n-\alpha},\infty}$ proved in conclusion (1),
we use (\ref{for proof, 5}) to deduce that
\begin{equation*}
  \begin{split}
    |\tilde{A}_{t,\rho}^{(1-\nu)\lambda,1}|
        \leq &
    \Big|\Big\{x\in B^c(0,\rho): \beta_t\frac{|\Omega(x)|}{|x|^{n-\alpha}}
    +
    \frac{2^{n+1-\alpha}}{|x|^{n-\alpha}}\int_{\mathbb{R}^n}|\Omega(x-y)-\Omega(x)|dV^1_t(y)>(1-\nu)\lambda\Big\}\Big|
    \\
    \leq &
    \Big|\Big\{x\in \mathbb{R}^n: \beta_t\frac{|\Omega(x)|}{|x|^{n-\alpha}}>(1-\nu)\lambda/2\Big\}\Big|
    \\
    + &
    \Big|\Big\{x\in B^c(0,\rho): \frac{2^{n+1-\alpha}}{|x|^{n-\alpha}}\int_{\mathbb{R}^n}|\Omega(x-y)-\Omega(x)|dV^1_t(y)>(1-\nu)\lambda/2\Big\}\Big|
    \\
    \leq &
    \left(\frac{2}{(1-\nu)\lambda}\beta_t\Big\|\frac{\Omega(\cdot)}{|\cdot|^{n-\alpha}}\Big\|_{L^{\frac{n}{n-\alpha},\infty}}\right)^{\frac{n}{n-\alpha}}
    +|\tilde{G}_{t,\rho}^{(1-\nu)\lambda/2^{n+1-\alpha},1}|
    \rightarrow 0\ \text{as}\ t\rightarrow 0^+.
  \end{split}
\end{equation*}
Recalling $\tilde{E}_t^{\nu\lambda,2}\rightarrow 0$ as $t\rightarrow 0^+$, we deduce that
\begin{equation*}
  \varlimsup_{t\rightarrow 0^+}|\tilde{A}_t^{\lambda}|
  \leq |B(0,\rho)|+\lim_{t\rightarrow 0^+}|\tilde{A}_{t,\rho}^{(1-\nu)\lambda,1}|+\lim_{t\rightarrow 0^+}|\tilde{E}_t^{\nu\lambda,2}|
  \leq |B(0,\rho)|,
\end{equation*}
which yields the desired conclusion by letting $\rho \rightarrow 0$.$\hfill\Box$

\medskip

{\it Proof of Theorem \ref{Theorem, fractional operator with homogeneous kernel, Lq-Dini}}.\quad Without loss of generality, we assume that $V$ is a probability measure. The conclusions (1) and (2) can be verified by the same method as in proof of Theorem
\ref{Theorem, fractional maximal with homogeneous kernel, Lq-Dini}.

To prove the conclusion (3), we employ the same notations $V_t^1,\,V_t^2$, $r_t$, $\epsilon_t$ for $t>0$, $\rho>0$, $\lambda>0$ as in the proof of Theorem \ref{Theorem, fractional operator with homogeneous kernel}.

Note that in the proof of Theorem \ref{Theorem, fractional operator with homogeneous kernel}, we have verified that
\begin{equation*}
  \begin{split}
    \left|T^{\alpha}_{\Omega}V^1_t(x)-\frac{\Omega(x)}{|x|^{n-\alpha}}\right|
    \leq
    \frac{2^{n-\alpha}}{|x|^{n-\alpha}}\left|\int_{\mathbb{R}^n}|\Omega(x-y)-\Omega(x)|dV^1_t(y)\right|
    +\frac{\beta_t|\Omega(x)|}{|x|^{n-\alpha}},
  \end{split}
\end{equation*}
where $\beta_t\rightarrow 0$ as $t\rightarrow 0^+$, for $x\in B^c(0,\rho)$.

Using the $L^{\frac{n}{n-\alpha}}$-Dini condition as in the proof of Theorem \ref{Theorem, fractional maximal with homogeneous kernel, Lq-Dini},
we obtain
\begin{equation*}
  \begin{split}
    \left\|\frac{1}{|\cdot|^{n-\alpha}}\int_{B(0,r_t)}|\Omega(\cdot-y)-\Omega(\cdot)|dV^1_t(y)\right\|_{L^{\frac{n}{n-\alpha},\infty}(\mathbb{R}^n\backslash B(0,\rho))}
       \rightarrow 0\ \text{as}\ t\rightarrow 0^+.
  \end{split}
\end{equation*}
Thus,
\begin{equation*}
  \begin{split}
    &\left\|T^{\alpha}_{\Omega}V^1_t(\cdot)-\frac{\Omega(\cdot)}{|\cdot|^{n-\alpha}}\right\|_{L^{\frac{n}{n-\alpha},\infty}(\mathbb{R}^n\backslash B(0,\rho))}
        \lesssim
    \beta_t\left\|\frac{|\Omega(\cdot)|}{|\cdot|^{n-\alpha}}\right\|_{L^{\frac{n}{n-\alpha},\infty}(\mathbb{R}^n)}
    \\
    &\qquad\qquad+
   \left\|\frac{2^{n-\alpha}}{|\cdot|^{n-\alpha}}\int_{B(0,r_t)}|\Omega(\cdot-y)-\Omega(\cdot)|dV^1_t(y)\right\|_{L^{\frac{n}{n-\alpha},\infty}(\mathbb{R}^n\backslash B(0,\rho))}\rightarrow 0\ \text{as}\ t\rightarrow 0^+.
  \end{split}
\end{equation*}
Recalling
\begin{equation*}
  \|T^{\alpha}_{\Omega}V^2_t(\cdot)\|_{L^{\frac{n}{n-\alpha},\infty}(\mathbb{R}^n)}\lesssim V_t^2(\mathbb{R}^n)=\epsilon_t\rightarrow 0\ \text{as}\ t\rightarrow 0^+,
\end{equation*}
we conclude that
\begin{equation*}
  \begin{split}
    \varlimsup_{t\rightarrow 0^+}
    \left\|T^{\alpha}_{\Omega}V_t(\cdot)-\frac{\Omega(\cdot)}{|\cdot|^{n-\alpha}}\right\|_{L^{\frac{n}{n-\alpha},\infty}(\mathbb{R}^n\backslash B(0,\rho))}
   &\lesssim
    \lim_{t\rightarrow 0^+}
    \left\|T^{\alpha}_{\Omega}V^1_t(\cdot)-\frac{\Omega(\cdot)}{|\cdot|^{n-\alpha}}\right\|_{L^{\frac{n}{n-\alpha},\infty}(\mathbb{R}^n\backslash B(0,\rho))}
    \\
    &\qquad\qquad+
    \lim_{t\rightarrow 0^+}\left\|T^{\alpha}_{\Omega}V^2_t(\cdot)\right\|_{L^{\frac{n}{n-\alpha},\infty}(\mathbb{R}^n)}=0,
  \end{split}
\end{equation*}
which completes the proof of (3). Theorem \ref{Theorem, fractional operator with homogeneous kernel, Lq-Dini} is proved. $\hfill\Box$

\medskip

{\it Proof of Corollary \ref{corollary, equivalent}.}\quad
  If $\Omega\in L^{\frac{n}{n-\alpha}}(\mathbb{S}^{n-1})$, we deduce
$\frac{\Omega(x)}{|x|^{n-\alpha}}\in L^{\frac{n}{n-\alpha},\infty}$
by Property (B). The conclusions (2) and (3) follow from
the weak Young's inequality (see Lemma \ref{Lemma, weak Young's inequality}) and the fact $M_{\Omega}^{\alpha}\leq T^{\alpha}_{|\Omega|}$.
The rest part of this proof follows directly from Theorems \ref{Theorem, fractional maximal with homogeneous kernel}
and \ref{Theorem, fractional operator with homogeneous kernel}. $\hfill\Box$

\medskip

By the similar method, one can also deal with the maximal truncated singular integral operator defined by
\begin{equation}
  T_{\Omega}^*(f)(x):=\sup_{\epsilon>0}\left|\int_{|x-y|>\epsilon}\frac{\Omega(x-y)}{|x-y|^n}f(y)dy\right|.
\end{equation}
We have the following two theorems.

\begin{theorem}\label{Theorem, maximal singular operator with homogeneous kernel}
Let $\alpha\in [0,n)$, $V$ be a absolutely continuous positive measure.
Suppose that $\Omega$ is a homogeneous function of degree zero and satisfies the $L^1_{\alpha}$-Dini condition.
If the maximal truncated singular integral operator $T_{\Omega}^{\ast}$
is bounded from $L^1$ to $L^{\frac{n}{n-\alpha},\infty}$, then
\begin{enumerate}
\item    $\frac{\Omega(x)}{|x|^{n-\alpha}}\in L^{\frac{n}{n-\alpha},\infty}$, $\Big\|\frac{\Omega(\cdot)}{|\cdot|^{n-\alpha}}\Big\|_{L^{\frac{n}{n-\alpha},\infty}}\lesssim \|T_{\Omega}^{\ast}\|_{L^1\rightarrow L^{\frac{n}{n-\alpha},\infty}}$;
\item    $\Omega\in L^{\frac{n}{n-\alpha}}(\mathbb{S}^{n-1})$, $\|\Omega \|_{L^{\frac{n}{n-\alpha}}(\mathbb{S}^{n-1})}\lesssim \|T_{\Omega}^{\ast}\|_{L^1\rightarrow L^{\frac{n}{n-\alpha},\infty}}$;
\item   $\displaystyle\lim_{t\rightarrow 0^{+}}\Big|\Big\{x\in \mathbb{R}^n: \Big|T^{\ast}_{\Omega}(V_t)(x)-\frac{|\Omega(x)|}{|x|^{\frac{n}{n-\alpha}}}V(\mathbb{R}^n)\Big|>\lambda\Big\}\Big|=0$, $\forall \lambda>0$.
\end{enumerate}
\end{theorem}

\begin{theorem}\label{Theorem, maximal singular operator with homogeneous kernel, Lq-Dini}
Let $\alpha\in [0,n)$, $V$ be a absolutely continuous positive measure.
Suppose that $\Omega$ is a homogeneous function of degree zero and satisfies the $L^{\frac{n}{n-\alpha}}$-Dini condition.
If the maximal truncated singular integral operator $T_{\Omega}^{\ast}$
is bounded from $L^1$ to $L^{\frac{n}{n-\alpha},\infty}$, we have
\begin{enumerate}
\item    $\frac{\Omega(x)}{|x|^{n-\alpha}}\in L^{\frac{n}{n-\alpha},\infty}$, $\Big\|\frac{\Omega(\cdot)}{|\cdot|^{n-\alpha}}\Big\|_{L^{\frac{n}{n-\alpha},\infty}}\lesssim \|T_{\Omega}^{\ast}\|_{L^1\rightarrow L^{\frac{n}{n-\alpha},\infty}}$;
\item    $\Omega\in L^{\frac{n}{n-\alpha}}(\mathbb{S}^{n-1})$, $\|\Omega \|_{L^{\frac{n}{n-\alpha}}(\mathbb{S}^{n-1})}\lesssim \|T_{\Omega}^{\ast}\|_{L^1\rightarrow L^{\frac{n}{n-\alpha},\infty}}$;
\item    $\displaystyle\lim_{t\rightarrow 0^{+}}\Big\|T^{\ast}_{\Omega}(V_t)-\frac{|\Omega(\cdot)|}{|\cdot|^{n-\alpha}}V(\mathbb{R}^n)\Big\|_{L^{\frac{n}{n-\alpha},\infty}(\mathbb{R}^n\backslash B(0,\rho))}=0$
    for every $\rho>0$.
\end{enumerate}
\end{theorem}

\medskip

{\it Proofs of Theorems \ref{Theorem, maximal singular operator with homogeneous kernel} and \ref{Theorem, maximal singular operator with homogeneous kernel, Lq-Dini}}.\quad These two theorems can be proved by
the similar argument as in the proofs of Theorems \ref{Theorem, fractional operator with homogeneous kernel} and \ref{Theorem, fractional operator with homogeneous kernel, Lq-Dini}. Here, we only give the following key estimates:

  Without loss of generality, we assume that $V$ is a probability measure. For $t>0$, let $V_t^1,\,V_t^2$, $\epsilon_t$ be as in the proof of Theorem \ref{Theorem, fractional operator with homogeneous kernel}.
  For every $\epsilon>0$, $x\in B(0,\rho)$, we have
  \begin{equation*}
    \begin{split}
      \left|\int_{|x-y|>\epsilon}\frac{\Omega(x-y)}{|x-y|^n}dV_t^1(y)\right|
          &\leq
      \left|\int_{|x-y|>\epsilon}\frac{\Omega(x)}{|x|^n}dV_t^1(y)\right|
      +
      \int_{|x-y|>\epsilon}\left|\frac{\Omega(x-y)}{|x-y|^n}-\frac{\Omega(x)}{|x|^n}\right|dV_t^1(y)
      \\
      &\leq
      \frac{|\Omega(x)|}{|x|^n}
      +
      \int_{\mathbb{R}^n}\left|\frac{\Omega(x-y)}{|x-y|^n}-\frac{\Omega(x)}{|x|^n}\right|dV_t^1(y),
    \end{split}
  \end{equation*}
  and then
  \begin{equation*}
    |T_{\Omega}^*(V_t)(x)|
    \leq
    \frac{|\Omega(x)|}{|x|^n}
    +
    \int_{\mathbb{R}^n}\left|\frac{\Omega(x-y)}{|x-y|^n}-\frac{\Omega(x)}{|x|^n}\right|dV_t^1(y).
  \end{equation*}
Also,
  \begin{equation*}
    \begin{split}
      \left|\int_{|x-y|>\epsilon}\frac{\Omega(x-y)}{|x-y|^n}dV_t^1(y)\right|
           \geq &
      \left|\int_{|x-y|>\epsilon}\frac{\Omega(x)}{|x|^n}dV_t^1(y)\right|
      -
      \int_{|x-y|>\epsilon}\left|\frac{\Omega(x-y)}{|x-y|^n}-\frac{\Omega(x)}{|x|^n}\right|dV_t^1(y)
      \\
      \geq &
      \left|\int_{|x-y|>\epsilon}\frac{\Omega(x)}{|x|^n}dV_t^1(y)\right|
      -
      \int_{\mathbb{R}^n}\left|\frac{\Omega(x-y)}{|x-y|^n}-\frac{\Omega(x)}{|x|^n}\right|dV_t^1(y),
    \end{split}
  \end{equation*}
  and then
  \begin{equation*}
    \begin{split}
      |T_{\Omega}^*(V_t)(x)|
      \geq &
      \sup_{\epsilon>0}\left|\int_{|x-y|>\epsilon}\frac{\Omega(x)}{|x|^n}dV_t^1(y)\right|
      -
      \int_{\mathbb{R}^n}\left|\frac{\Omega(x-y)}{|x-y|^n}-\frac{\Omega(x)}{|x|^n}\right|dV_t^1(y)
      \\
      = &
      \frac{|\Omega(x)|}{|x|^n}\int_{\mathbb{R}^n}dV_t^1(y)
      -
      \int_{\mathbb{R}^n}\left|\frac{\Omega(x-y)}{|x-y|^n}-\frac{\Omega(x)}{|x|^n}\right|dV_t^1(y)
      \\
      = &
      \frac{(1-\epsilon_t)|\Omega(x)|}{|x|^n}
      -
      \int_{\mathbb{R}^n}\left|\frac{\Omega(x-y)}{|x-y|^n}-\frac{\Omega(x)}{|x|^n}\right|dV_t^1(y).
    \end{split}
  \end{equation*}
The other details are omitted here.$\hfill\Box$

\section{weak Young's inequality}
As we know, if $\Omega\in L^{\frac{n}{n-\alpha}}(\mathbb{S}^{n-1})$ for $\alpha\in (0,n)$, then the $L^1\rightarrow L^{\frac{n}{n-\alpha}}$ boundedness of $T_{\Omega}^{\alpha}$
can be deduced by the following weak Young' inequality (see \cite[Theorem 1.2.13]{Grafakos_Classical_2008}).
\begin{lemma}(cf. \cite{Grafakos_Classical_2008})\label{Lemma, weak Young's inequality}
  Let $1\leq p<\infty$ and $1<q,r<\infty$ satisfy
  \begin{equation*}
    1+\frac{1}{q}=\frac{1}{p}+\frac{1}{r}.
  \end{equation*}
  Then for $f\in L^p(\mathbb{R}^n)$, $g\in L^{r,\infty}(\mathbb{R}^n)$,
  \begin{equation*}
    \|f\ast g\|_{L^{q,\infty}}\lesssim \|g\|_{L^{r,\infty}}\|f\|_{L^p}.
  \end{equation*}
\end{lemma}
Suppose $\Omega\in L^{\frac{n}{n-\alpha}}(\mathbb{S}^{n-1})$, then $\frac{\Omega(x)}{|x|^{n-\alpha}}\in L^{\frac{n}{n-\alpha},\infty}(\mathbb{R}^n)$.
Take $g(x)=\frac{\Omega(x)}{|x|^{n-\alpha}}$ in Lemma \ref{Lemma, weak Young's inequality},
then the $L^1\rightarrow L^{\frac{n}{n-\alpha},\infty}$ boundedness of $T_{\Omega}^{\alpha}$ follows.
Inspired by Theorem \ref{Theorem, fractional operator with homogeneous kernel},
a natural question is: if we replace $\frac{\Omega(x)}{|x|^{n-\alpha}}$ by a more general function $g$ and keep the boundedness for
the map $T_g(V)(x):= \int_{\mathbb{R}^n}g(x-y)dV(y)$,
what can we say for $g$, and how about the limiting behavior of $T_g(V_t)$?

The main purpose of this section is to answer the above question. Firstly, we consider the case of that $g$ is a radial function.

\begin{theorem}\label{Theorem, weak Young, radial}
  Let $\alpha\in [0,n)$, $r\in (1,\infty)$, $V$ be a absolutely continuous positive measure.
  Suppose that $g$ is a positive radial function and is decreasing in the radial direction. If $T_g$
is bounded from $L^1$ to $L^{r,\infty}$, then
\begin{enumerate}
\item    $g\in L^{r,\infty}$, $\|g\|_{L^{r,\infty}}\lesssim \|T_g\|_{L^1\rightarrow L^{r,\infty}}$;
\item    $\displaystyle\varliminf_{t\rightarrow 0^{+}}|\{x\in \mathbb{R}^n: |T_gV_t(x)|>\lambda\}|\geq |\{x\in \mathbb{R}^n: |g(x)|V(\mathbb{R}^n)>\lambda\}|$, $\forall \lambda>0$;
\item    $\displaystyle\varlimsup_{t\rightarrow 0^{+}}|\{x\in \mathbb{R}^n: |T_gV_t(x)|>\lambda\}|\leq |\{x\in \mathbb{R}^n: |g(x)|V(\mathbb{R}^n)\geq\lambda\}|$, $\forall \lambda>0$.
\end{enumerate}
\end{theorem}

\begin{proof}
Without loss of generality, we assume that $V(\mathbb{R}^n)=1$. For $t>0$, let $V_t^1,\,V_t^2$, $r_t$, $\epsilon_t$ be as before. For $\lambda>0$, denote
\begin{equation*}
  K_{t}^{\lambda}:=\{x\in \mathbb{R}^n:\, T_gV_t(x)>\lambda\},
\end{equation*}
and
\begin{equation*}
  K_{t}^{\lambda,1}:=\{x\in \mathbb{R}^n: |T_gV^1_t(x)|>\lambda\},\hspace{6mm}K_{t}^{\lambda,2}:=\{x\in \mathbb{R}^n: |T_gV^2_t(x)|>\lambda\}.
\end{equation*}
For a fixed $\nu>0$, recalling that $T_g$ is boundedness from $L^1$ to $L^{r,\infty}$, we obtain that
\begin{equation*}
  (\nu\lambda)|K_{t}^{\nu\lambda,2}|^{\frac{1}{r}}\lesssim V^2_t(\mathbb{R}^n)= \epsilon_t \rightarrow 0^+\ \text{as}\ t \rightarrow 0^+,
\end{equation*}
which implies $|K_{t}^{\nu\lambda,2}|\rightarrow 0$ as $t\rightarrow 0^+$.
For a fixed $\rho>0$, set
\begin{equation*}
  K_{t,\rho}^{\lambda,1}=K_{t}^{\lambda,1}\cap B^c(0,\rho).
\end{equation*}
Then
\begin{equation*}
  \begin{split}
    K_{t}^{\lambda}
    \supset
    K_{t}^{(1+\nu)\lambda,1}\backslash K_{t}^{\nu\lambda,2}
       \supset
    K_{t,\rho}^{(1+\nu)\lambda,1}\backslash K_{t}^{\nu\lambda,2}.
  \end{split}
\end{equation*}
For $\rho>2r_t$, $x\in B^c(0,\rho)$,
\begin{equation*}
  \begin{split}
    T_gV^1_t(x)
    &=    \int_{B(0,r_t)}g(x-y)dV^1_t(y)
        \geq
    \int_{B(0,r_t)}g(\frac{x(|x|+r_t)}{|x|})dV^1_t(y)
    \\
    &\geq
    \int_{B(0,r_t)}g(\frac{x(\rho+r_t)}{\rho})dV^1_t(y)
    =(1-\epsilon_t)g(\frac{x(\rho+r_t)}{\rho}),
  \end{split}
\end{equation*}
and
\begin{equation*}
  \begin{split}
    T_gV^1_t(x)
        &\leq
    \int_{B(0,r_t)}g(\frac{x(|x|-r_t)}{|x|})dV^1_t(y)
        \leq
    \int_{B(0,r_t)}g(\frac{x(\rho-r_t)}{\rho})dV^1_t(y)
    =(1-\epsilon_t)g(\frac{x(\rho-r_t)}{\rho}).
  \end{split}
\end{equation*}
Thus, for $\rho>2r_t$, $x\in B^c(0,\rho)$,
\begin{equation*}
  \begin{split}
    |T_gV^1_t(x)|
    \leq &
    (1-\epsilon_t)\max\Big\{\Big|g(\frac{x(\rho+r_t)}{\rho})\Big|,\,\Big|g(\frac{x(\rho-r_t)}{\rho})\Big|\Big\}
    \\
    = &
    (1-\epsilon_t)\Big|g(\frac{x(\rho-r_t)}{\rho})\Big|
  \end{split}
\end{equation*}
and
\begin{equation*}
  \begin{split}
    |T_gV^1_t(x)|
    \geq &
    (1-\epsilon_t)\min\Big\{\Big|g(\frac{x(\rho+r_t)}{\rho})\Big|,\,\Big|g(\frac{x(\rho-r_t)}{\rho})\Big|\Big\}
    \\
    = &
    (1-\epsilon_t)\Big|g(\frac{x(\rho+r_t)}{\rho})\Big|.
  \end{split}
\end{equation*}
Denote
\begin{equation}
  L_{t,\rho}^{\lambda}:=\Big\{x\in B^c(0,\rho): (1-\epsilon_t)\Big|g(\frac{x(\rho+r_t)}{\rho})\Big|>\lambda\Big\}.
\end{equation}
Observing $L_{t,\rho}^{\lambda}\subset K_{t,\rho}^{\lambda,1}$ and
recalling
$K_{t}^{\lambda} \supset K_{t,\rho}^{(1+\nu)\lambda,1}\backslash K_{t}^{\nu\lambda,2}$,
we deduce that
\begin{equation*}
  \begin{split}
    |K_{t}^{\lambda}|
    \geq
    |K_{t,\rho}^{(1+\nu)\lambda,1}|-|K_{t}^{\nu\lambda,2}|
    \geq
    |L_{t,\rho}^{(1+\nu)\lambda}|-|K_{t}^{\nu\lambda,2}|.
  \end{split}
\end{equation*}
Recalling $|K_{t}^{\nu\lambda,2}|\rightarrow 0$ as $t\rightarrow 0^+$, we have
\begin{equation*}
  \begin{split}
    \varliminf_{t\rightarrow 0^{+}}|K_{t}^{\lambda}|
    \geq &
    \varliminf_{t\rightarrow 0^{+}}|L_{t,\rho}^{(1+\nu)\lambda}|
     \geq
    \varliminf_{t\rightarrow 0^{+}}\Big|\Big\{x\in \mathbb{R}^n: (1-\epsilon_t)\Big|g(\frac{x(\rho+r_t)}{\rho})\Big|>(1+\nu)\lambda\Big\}\Big|-|B(0,\rho)|
    \\
    = &
    |\{x\in \mathbb{R}^n: |g(x)|>(1+\nu)\lambda\}|-|B(0,\rho)|.
  \end{split}
\end{equation*}
Letting $\nu\rightarrow 0$ and $\rho\rightarrow 0$, we obtain
\begin{equation}\label{Theorem, weak Young, radial, for proof 1}
  \varliminf_{t\rightarrow 0^{+}}|K_{t}^{\lambda}|\geq |\{x\in \mathbb{R}^n: |g(x)|>\lambda\}|.
\end{equation}
Recalling the definition of $K_t^{\lambda}$ and the boundedness of $T_g$, we obtain
\begin{equation*}
  \begin{split}
    \lambda|\{x\in \mathbb{R}^n: |g(x)|>\lambda\}|^{\frac{1}{r}}
    \leq
    \varliminf_{t\rightarrow 0^{+}}\lambda|K_{t}^{\lambda}|^{\frac{1}{r}}
       \leq
    \varliminf_{t\rightarrow 0^{+}}\|T_gV_t\|_{L^{r,\infty}}
     \leq
    \|T_g\|_{L^1\rightarrow L^{r,\infty}}.
  \end{split}
\end{equation*}
By the arbitrary of $\lambda$, we deduce
$$g\in L^{r,\infty},\quad {\rm and}\quad \|g\|_{L^{r,\infty}}\leq \|T_g\|_{L^1\rightarrow L^{r,\infty}},$$
which completes the proof of conclusion (1).

Set
\begin{equation*}
  \widetilde{L}_{t,\rho}^{\lambda}=\Big\{x\in B^c(0,\rho): (1-\epsilon_t)\Big|g(\frac{x(\rho-r_t)}{\rho})\Big|>\lambda\Big\}.
\end{equation*}
Observing $\widetilde{L}_{t,\rho}^{\lambda}\supset K_{t,\rho}^{\lambda,1}$
and $K_{t}^{\lambda} \subset K_{t,\rho}^{(1-\nu)\lambda,1}\cup K_{t}^{\nu\lambda,2}\cup B(0,\rho)$,
we deduce that
\begin{equation*}
  \begin{split}
    |K_{t}^{\lambda}|
    \leq
    |K_{t,\rho}^{(1-\nu)\lambda,1}|+|K_{t}^{\nu\lambda,2}|+|B(0,\rho)|
    \leq
    |\widetilde{L}_{t,\rho}^{(1-\nu)\lambda}|+|K_{t}^{\nu\lambda,2}|+|B(0,\rho)|.
  \end{split}
\end{equation*}
Recalling $|K_{t}^{\nu\lambda,2}|\rightarrow 0$ as $t\rightarrow 0^+$, we have
\begin{equation*}
  \begin{split}
    \varlimsup_{t\rightarrow 0^{+}}|K_{t}^{\lambda}|
    \leq &
    \varlimsup_{t\rightarrow 0^{+}}|\widetilde{L}_{t,\rho}^{(1-\nu)\lambda}|+|B(0,\rho)|
    \\
    \leq &
    \varlimsup_{t\rightarrow 0^{+}}\Big|\Big\{x\in \mathbb{R}^n: (1-\epsilon_t)\Big|g(\frac{x(\rho-r_t)}{\rho})\Big|>(1-\nu)\lambda\Big\}\Big|+|B(0,\rho)|
    \\
    = &
    |\{x\in \mathbb{R}^n: |g(x)|>(1-\nu)\lambda\}|+|B(0,\rho)|.
  \end{split}
\end{equation*}
Letting $\nu\rightarrow 0$ and $\rho\rightarrow 0$, we obtain
\begin{equation*}\label{Theorem, weak Young, radial, for proof 2}
  \varlimsup_{t\rightarrow 0^{+}}|K_{t}^{\lambda}|\leq |\{x\in \mathbb{R}^n: |g(x)|\geq\lambda\}|.
\end{equation*}
This together with (\ref{Theorem, weak Young, radial, for proof 1}) yields the conclusions
(2) and (3). Theorem \ref{Theorem, weak Young, radial} is proved.
\end{proof}

The following result can be deduced by weak Young's inequality (see Lemma \ref{Lemma, weak Young's inequality}) and Theorem \ref{Theorem, weak Young, radial}.

\begin{corollary}\label{Corollary, weak Young, radial}
  Let $r\in (1,\infty)$.
  Suppose that $g$ is a positive radial function which is decreasing in the radial direction.
  Then the following two statements are equivalent.
  \begin{enumerate}
  \item $g\in L^{r,\infty}$.
  \item $T_g$ is bounded from $L^1$ to $L^{r,\infty}$.
\end{enumerate}
\end{corollary}

If we drop the assumption of radial and deceasing, add some integrability, the limiting behavior
can be established as follows.

\begin{theorem}\label{Theorem, weak Young, L1}
  Let $r\in (1,\infty)$, $V$ be an absolutely continuous positive measure.
  If $T_g$ is bounded from $L^1$ to $L^{r,\infty}$ for some $g\in L^1(\mathbb{R}^n)$, then
\begin{enumerate}
\item    $g\in L^{r,\infty}$, $\|g\|_{L^{r,\infty}}\lesssim \|T_g\|_{L^1\rightarrow L^{r,\infty}}$;
\item    $\displaystyle\lim_{t\rightarrow 0^{+}}|\{x\in \mathbb{R}^n: |T_g(V_t)(x)-g(x)|>\lambda\}|=0$
    for every $\lambda>0$.
\end{enumerate}
\end{theorem}

\begin{proof}
Without loss of generality, we assume that $V(\mathbb{R}^n)=1$. For $t>0$, $\lambda>0$, let $V_t^1,\,V_t^2$, $r_t$, $\epsilon_t$, $K_t^\lambda$, $K_t^{\lambda,1}$, $K_t^{\lambda,2}$ be as in the proof of Theorem \ref{Theorem, weak Young, radial}. Then
\begin{equation*}
  \begin{split}
    K_{t}^{\lambda}
    \supset &
    K_{t}^{(1+\nu)\lambda,1}\backslash K_{t}^{\nu\lambda,2}.
  \end{split}
\end{equation*}
For a fixed $\nu>0$, recalling that $T_g$ is boundedness from $L^1$ to $L^{r,\infty}$, we obtain that
\begin{equation*}
  (\nu\lambda)|K_{t}^{\nu\lambda,2}|^{\frac{1}{r}}\lesssim V^2_t(\mathbb{R}^n)= \epsilon_t \rightarrow 0^+\ \text{as}\ t \rightarrow 0^+.
\end{equation*}
For $\rho>2r_t$, $x\in B^c(0,\rho)$,
\begin{equation*}
  \begin{split}
    |T_gV^1_t(x)|
        = &
    \left|\int_{B(0,r_t)}g(x-y)dV^1_t(y)\right|
    \\
    = &
    \left|\int_{B(0,r_t)}g(x-y)-g(x)dV^1_t(y)+(1-\epsilon_t)g(x)\right|
    \\
    \geq &
    (1-\epsilon_t)|g(x)|
    -\left|\int_{B(0,r_t)}g(x-y)-g(x)dV^1_t(y)\right|.
  \end{split}
\end{equation*}
Set
\begin{equation*}
  R_{t}^{\lambda}=\{x\in \mathbb{R}^n:\, (1-\epsilon_t)|g(x)|>\lambda\}
\end{equation*}
and
\begin{equation*}
  S_{t}^{\lambda,1}=\Big\{x\in \mathbb{R}^n:\, \Big|\int_{B(0,r_t)}(g(x-y)-g(x))dV^1_t(y)\Big|>\lambda\Big\}.
\end{equation*}
Now, we estimate $|S_{t}^{\nu\lambda,1}|$ for fixed $\nu>0$.
Using Chebychev's inequality, we conclude that
\begin{equation*}
  \begin{split}
    |S_{t}^{\nu\lambda,1}|
    \leq &
    \frac{1}{\nu\lambda}\int_{\mathbb{R}^n}\Big|\int_{B(0,r_t)}(g(x-y)-g(x))dV^1_t(y)\Big|dx
    \\
    \leq &
    \frac{1}{\nu\lambda}\int_{B(0,r_t)}\int_{\mathbb{R}^n}|g(x-y)-g(x)|dx dV^1_t(y)
    \\
    \leq &
    \frac{1}{\nu\lambda}\sup_{y\in B(0,r_t)}\int_{\mathbb{R}^n}|g(x-y)-g(x)|dx\rightarrow 0\ \text{as}\ t\rightarrow 0^+,
  \end{split}
\end{equation*}
where we use the average continuous of for $g\in L^1$.

Observing $K_{t}^{(1+\nu)\lambda,1}\supset R_{t}^{(1+2\nu)\lambda}\backslash S_{t}^{\nu\lambda,1}$ and recalling
$K_{t}^{\lambda} \supset K_{t}^{(1+\nu)\lambda,1}\backslash K_{t}^{\nu\lambda,2}$, we deduce that
\begin{equation*}
  \begin{split}
    |K_{t}^{\lambda}|
    \geq
    |K_{t}^{(1+\nu)\lambda,1}|-|K_{t}^{\nu\lambda,2}|
    \geq
    |R_{t}^{(1+2\nu)\lambda}|-|S_{t}^{\nu\lambda,1}|-|K_{t}^{\nu\lambda,2}|.
  \end{split}
\end{equation*}
Recalling $|K_{t}^{\nu\lambda,2}|$, $|S_{t}^{\nu\lambda,1}|\rightarrow 0$ as $t\rightarrow 0^+$, we have
\begin{equation*}
  \begin{split}
    \varliminf_{t\rightarrow 0^{+}}|K_{t}^{\lambda}|
    \geq &
    \varliminf_{t\rightarrow 0^{+}}|R_{t}^{(1+2\nu)\lambda}|
        =
    \varliminf_{t\rightarrow 0^{+}}|\{x\in \mathbb{R}^n: (1-\epsilon_t)|g(x)|>(1+2\nu)\lambda)\}|
    \\
    = &
    |\{x\in \mathbb{R}^n: |g(x)|>(1+2\nu)\lambda)\}|.
  \end{split}
\end{equation*}
Letting $\nu\rightarrow 0$, we obtain
\begin{equation}
  \varliminf_{t\rightarrow 0^{+}}|K_{t}^{\lambda}|\geq |\{x\in \mathbb{R}^n: |g(x)|>\lambda\}|.
\end{equation}
Recalling the definition of $K_t^{\lambda}$ and the boundedness of $T_g$, we obtain
\begin{equation*}
  \begin{split}
    \lambda|\{x\in \mathbb{R}^n: |g(x)|>\lambda\}|^{\frac{1}{r}}
    \leq
    \varliminf_{t\rightarrow 0^{+}}\lambda|K_{t}^{\lambda}|^{\frac{1}{r}}
    =
    \varliminf_{t\rightarrow 0^{+}}\|T_gV_t\|_{L^{r,\infty}}
        \leq
    \|T_g\|_{L^1\rightarrow L^{r,\infty}}.
  \end{split}
\end{equation*}
Since $\lambda>0$ is arbitrary, we deduce
$g\in L^{r,\infty}$
and $\|g\|_{L^{r,\infty}}\leq \|T_g\|_{L^1\rightarrow L^{r,\infty}}$, which complets the proof of conclusion (1).

Next, we turn to verify conclusion (2). For $x\in \mathbb{R}^n$, we have
\begin{equation*}
  \begin{split}
    |T_gV^1_t(x)-g(x)|
    = &
    \left|\int_{B(0,r_t)}(g(x-y)-g(x))dV^1_t(y)-\epsilon_tg(x)\right|
    \\
    \leq &
    \epsilon_t|g(x)|
    +\left|\int_{B(0,r_t)}(g(x-y)-g(x))dV^1_t(y)\right|
  \end{split}
\end{equation*}
Denote
\begin{equation*}
  Q_{t}^{\lambda}=\{x\in \mathbb{R}^n: |T_gV_t(x)-g(x)|>\lambda\},
\end{equation*}
\begin{equation*}
  Q_{t}^{\lambda,1}=\{x\in \mathbb{R}^n: |T_gV^1_t(x)-g(x)|>\lambda\}.
\end{equation*}
Then
\begin{equation*}
  \begin{split}
   Q_t^{\lambda}
    \subset
    Q_{t}^{(1-\nu)\lambda,1}\cup K_t^{\nu\lambda,2}.
  \end{split}
\end{equation*}
A direct calculation yields that
\begin{equation*}
  \begin{split}
    |Q_{t}^{(1-\nu)\lambda,1}|
      & \leq
    |\{x\in \mathbb{R}^n: |T_gV^1_t(x)-g(x)|>(1-\nu)\lambda\}|
    \\
    &\leq
    |\{x\in \mathbb{R}^n: \epsilon_t|g(x)|>(1-\nu)\lambda/2\}|
    \\
    &\ \ \ \ +
    \Big|\Big\{x\in \mathbb{R}^n:\, \Big|\int_{B(0,r_t)}g(x-y)-g(x)dV^1_t(y)\Big|>(1-\nu)\lambda/2\Big\}\Big|
    \\
    &\leq
    \left(\frac{2}{(1-\nu)\lambda}\epsilon_t\|g\|_{L^{r,\infty}}\right)^r
    +|S_{t}^{(1-\nu)\lambda/2,1}|
    \rightarrow 0\ \text{as}\ t\rightarrow 0^+.
  \end{split}
\end{equation*}
Recalling $|K_t^{\nu\lambda,2}|\rightarrow 0$ as $t\rightarrow 0^+$, we deduce that
\begin{equation*}
  \varlimsup_{t\rightarrow 0^+}|Q_t^{\lambda}|
  \leq
  \varlimsup_{t\rightarrow 0^+}|Q_{t}^{(1-\nu)\lambda,1}|
  +\varlimsup_{t\rightarrow 0^+}|K_t^{\nu\lambda,2}|=0,
\end{equation*}
which is the desired conclusion.
\end{proof}

\begin{remark}
  Note that Proposition \ref{proposition, comparison for limiting behavior} does not work in the case of Theorem \ref{Theorem, weak Young, L1},
  since the set $\{x\in \mathbb{R}^n: |g(x)|=\lambda\}$ may has positive measure for some $\lambda>0$.
  Thus, we can not deduce type-3 convergence from the conclusion of type-2 convergence obtained in Theorem \ref{Theorem, weak Young, L1}.
\end{remark}

\begin{corollary}\label{Corollary, weak Young, L1}
  Let $r\in (1,\infty)$.
  Suppose $g\in L^1(\mathbb{R}^n)$. Then the following two statements are equivalent.
  \begin{enumerate}
  \item $g\in L^{r,\infty}$.
  \item $T_g$ is bounded from $L^1$ to $L^{r,\infty}$.
\end{enumerate}
\end{corollary}

This corollary follows immediately from the weak Young's inequality (see Lemma \ref{Lemma, weak Young's inequality}) and Theorem \ref{Theorem, weak Young, L1}. We omit the details here.


\begin{thebibliography}{99}
\bibitem{DL1} Y. Ding and X. Lai,
Weak type $(1, 1)$ behavior for the maximal operator with $L^1$-Dini kernel,
Potential Anal. 47(2) (2017), 169-187.

\bibitem{DL2} Y. Ding and X. Lai,
$L^1$-Dini conditions and limiting behavior of weak type estimates for singular integrals,
Rev. Mat. Iberoam. (to appear).

\bibitem{Grafakos_Classical_2008} L. Grafakos,
Classical Fourier Analysis, Graduate Texts in Mathematics, 2008, 249.

\bibitem{Jan_Trans_2006} P. Janakiraman,
Limiting weak-type behavior for singular integral and maximal operators,
Trans. Amer. Math. Soc. 358 (2006), 1937-1952.

\end{thebibliography}
\end{document}